\documentclass[12pt]{amsart}
\usepackage{epsfig,psfrag,verbatim,amssymb,amsmath,amsfonts,graphicx,bbm}
\newtheorem{theorem}{Theorem}
\newtheorem{lemma}[theorem]{Lemma}
\newtheorem{prop}[theorem]{Proposition}
\newtheorem{cor}[theorem]{Corollary}

\newtheorem{remark}[theorem]{Remark}

\newtheorem{problem}[theorem]{Problem}

\newcommand{\EE}{\mathbb{E}}
\newcommand{\I}{\mathbbm{1}}
\newcommand{\R}{\mathbb{R}}

\newcommand{\Z}{\mathbb{Z}}
\newcommand{\N}{\mathbb{N}}

\newcommand{\PP}{\mathbb{P}}

\newcommand{\hY}{\overline Y}
\newcommand{\hV}{\overline V}

\newcommand{\F}{\mathcal{F}}
\newcommand{\la} {\lambda}
\newcommand{\si}{\sigma}
\newcommand{\al}{\alpha}
\newcommand{\ga}{\gamma}
\newcommand{\om}{\omega}
\newcommand{\Om}{\Omega}
\newcommand{\eps}{\varepsilon}

\renewcommand{\epsilon}{\varepsilon}

\newcounter{constante}
\newcommand{\con}[1]{
\immediate\write 1{\noexpand\newlabel{#1}{{\theconstante}{\theconstante}}}
                    c_{\theconstante}
                    \stepcounter{constante}
                   }

\begin{document}

\title[Excited random walks] {Excursions of excited random walks on integers}

\author{Elena Kosygina and Martin P.W.\ Zerner} 


\begin{abstract}
Several phase transitions for excited random walks on the integers are known to be characterized by a certain drift parameter $\delta\in\R$. For recurrence/transience the critical threshold is $|\delta|=1$, for ballisticity it is $|\delta|=2$ and for diffusivity $|\delta|=4$. In this paper we establish a phase transition at $|\delta|=3$. We show that the  expected return time of the walker to the starting point, conditioned on return, is finite iff $|\delta|>3$.  This result follows from an explicit description of the tail behaviour of the return time as a function of $\delta$, which is achieved by diffusion approximation of related branching processes by squared Bessel processes.
\end{abstract}
\maketitle
\section{Introduction}
A transient random walk (RW) is called {\em strongly transient} if the expectation of its return time $R$
to the starting point, conditioned on $R<\infty$, is finite, see e.g.\ \cite[\S 3.2.6]{Hug95}  and the references therein. The simple symmetric RW on $\Z^d$ is strongly transient iff $d\ge 5$, see \cite[\S 3.3.4, Table 3.4]{Hug95}.  ``Under fairly general conditions, biased walks are strongly transient'' \cite[p.\ 127]{Hug95}. In the present paper we study the tail behavior of the depth and the duration of excursions of excited random walks (ERWs). In particular, we show that ERWs can be biased, in the sense of satisfying a strong law of large numbers with non-zero speed,  and at the same time be not strongly transient.   
Precise statements are given later in this section after we describe our model of ERW. (For a recent survey on ERW we refer the reader to \cite{KZ13}.) 

An ERW evolves in a so-called cookie environment. These are elements $\om=(\om(z,i))_{z\in\Z, i\ge 1}$ of
$\Om:=[0,1]^{\Z\times \N}$.
Given $\om\in\Om, z\in\Z$ and  $i\in\N$ we call $\om(z,i)$ the $i$-th cookie at site $z$ and $\om(z,\cdot)$ the stack of cookies at $z$. The cookie $\om(z,i)$ serves as transition probability from $z$ to $z+1$ of the ERW upon its $i$-th visit to $z$. More precisely, given $\om\in\Om$ and $x\in\Z$ an ERW starting at $x$ in the environment $\om$ is a process $(X_n)_{n\ge 0}$ on a suitable probability space $(\Om',\F',P_{x,\om})$ which satisfies for all $n\ge 0$:
\begin{eqnarray}\nonumber
  P_{x,\omega}[X_0=x]&=&1,\\
P_{x,\omega}[X_{n+1}=X_n+1\,|(X_i)_{0\le i\le n}]&=&
  \omega(X_n,\#\{i\le n\,|X_i=X_n\}),\label{erw}\\ 
  P_{x,\omega}[X_{n+1}=X_n-1\,|(X_i)_{0\le i\le n}]&=&
  1-\omega(X_n,\#\{i\le n\,|X_i=X_n\}).\nonumber
\end{eqnarray}
The environment $\om$ is chosen at random according to some probability measure $\PP$ on $(\Om,\F)$, where $\F$ is the canonical product Borel $\si$-field. Throughout the paper we assume that $\PP$ satisfies the following hypotheses (IID), (WEL), and (BD$_\mathrm{M}$) for some $M\in\N_0:=\N\cup\{0\}$.
\begin{equation*}
\mbox{The family $(\omega(z,\cdot))_{z\in\Z}$ of cookie stacks is i.i.d.\ under
$\PP$.}
\tag{IID}
\end{equation*}
We denote the distribution of $\omega(0,\cdot)$ under $\PP$ by $\nu$, so that $\PP=\bigotimes_{\Z}\nu$. 
To avoid degenerate cases we assume the following (weak) ellipticity hypothesis: 
\begin{equation*}
\PP\left[\forall i\in\N:\ \om(z,i)>0\right]>0,
\PP\left[\forall i\in\N:\ \om(z,i)<1\right]>0\quad\mbox{for all $z\in\Z$.}\tag{WEL}
\end{equation*}
If we assumed  only (IID) and (WEL) the model would include RWs in random i.i.d.\ environments (RWRE), since for them $\PP$-a.s.\ $\om(0,i)=\om(0,1)$ for all $i\ge 1$. However, for the ERW model considered in this paper we assume that there is a non-random $M\ge 0$ such that after $M$ visits to any site the ERW behaves on any subsequent visit to that site like a simple symmetric RW:
\begin{equation*}
\mbox{$\PP$-a.s.\ $\om(z,i)=1/2$  for all $z\in\Z$ and $i>M$.}
\tag{BD$_\mathrm{M}$} 
\end{equation*}
If we average the so-called {\em quenched} measure $P_{x,\om}$ defined above over the environment $\om$ we obtain the {\em averaged} (often also called {\em annealed}) measure $P_x[\cdot]:=\EE[P_{x,\om}[\cdot]]$ on $\Om\times\Om'$. The expectation operators with respect to $P_{x,\om}, \PP,$ and $P_x$ are denoted by 
$E_{x,\om}, \EE,$ and $E_x$, respectively.

Several features of the ERW can be characterized by the parameter
\begin{equation}\label{del}
\delta:=\EE\left[\sum_{i\ge 1}(2\om(0,i)-1)\right]=
\EE\left[\sum_{i=1}^M(2\om(0,i)-1)\right],
\end{equation}
which represents the expected total average displacement of the ERW after consumption of all the cookies at any given site. Most prominently, the ERW $(X_n)_{n\ge 0}$
\begin{itemize}
\item is transient, i.e.\ tends $P_0$-a.s.\ to $\pm\infty$,  iff $|\delta|> 1$ (see 
\cite[Th.\ 3.10]{KZ13} and the references therein),
\item is ballistic, i.e.\ has $P_0$-a.s.\ a deterministic non-zero speed $\lim_{n\to\infty}X_n/n$,  iff $|\delta|>2$ (see \cite[Th.\ 5.2]{KZ13} and the references therein),
\item converges after diffusive scaling under $P_0$ to a Brownian motion iff $|\delta|>4$ or $\delta=0$ (see \cite[Theorems 6.1, 6.3, 6.5, 6.6, 6.7]{KZ13} and the references therein).
\end{itemize}
In this paper we are concerned with the finite excursions of ERWs.
Let 
\[R:=\inf\{n\ge 1:\ X_n=X_0\}\]
be the time at which the RW returns to its starting point.
Denote for $k\in\Z$ the first passage time of $k$ by
\[T_k:=\inf\{n\ge 0:  X_n=k\}.\]
\begin{theorem}\label{ex}{\bf (Averaged excursion depth, duration, and return time)}
Let $\delta\in\R\backslash\{1\}$.
There are constants $\con{m0}(\nu),\con{t0}(\nu),\con{du}(\nu)\in (0,\infty)$ such that 
\begin{eqnarray}\label{M0}
\lim_{n\to\infty}n^{|\delta-1|}\,P_1[T_n<T_0<\infty]&=&c_{\ref{m0}},\\
\lim_{n\to\infty}n^{|\delta-1|/2}\,P_1[n<T_0<\infty]&=&c_{\ref{t0}},\label{T0}\\
\label{tu}
\lim_{n\to\infty}n^{||\delta|-1|/2}\, P_0[n<R<\infty]&=&c_{\ref{du}}.
\end{eqnarray}
Moreover, for $\delta=1$ and every $\eps>0$,
\begin{equation}\label{3s}
\lim_{n\to\infty}n^\eps\,P_1[T_n<T_0]=
\lim_{n\to\infty}n^\eps\,P_1[T_0>n]=
\lim_{n\to\infty}n^\eps\,P_0[R>n]=\infty.
\end{equation}
\end{theorem}
An immediate consequence of (\ref{tu}) and (\ref{3s}) is the following result.
\begin{cor}\label{dog}{\bf (Averaged strong transience)}
$E_0[R, R<\infty]<\infty$ iff $|\delta|>3$.
\end{cor}
\begin{remark}{\rm {\bf (Case $\delta=1$)} Relations (\ref{3s}) are an easy consequence of (\ref{M0})-(\ref{tu}) (see the proof in Section \ref{main}). We believe that for $\delta=1$ the quantities
$P_1[T_n<T_0], P_1[T_0>n]$, and $P_0[R>n]$ have a logarithmic decay.
In the special case described in  Remark \ref{s1} below, the existence of a nontrivial limit of $(\ln n)P_1[T_n<T_0]$ as $n\to\infty$ follows from connections with branching processes with immigration
and \cite[second part of (21)]{Zub72}, see also \cite[Th.\ 1, part 2]{FYK90}, quoted in \cite[Th.\ A (ii)]{KZ08}. }
\end{remark}
\begin{remark}{\rm {\bf (Once-excited RWs)}
In the case of once-excited RWs  with identical cookies (i.e.\ $M=1$, $\PP$-a.s.\ $\om(z,1)=\om(0,1)\in(0,1)$ for all $z\in\Z$), results (\ref{M0}) and (\ref{T0}) have been obtained in \cite[Section 3.3]{AR05}. Note that the case $M=1$ is very special, since at time $T_k$ all the cookies $\om(z,i)\ne 1/2$ between the starting point 0 and the current location $k$ of the ERW  have been ``eaten''. This allows to use simple symmetric RW calculations between 0 and $k$. For $M\ge 2$ such simplification is not available.}
\end{remark}
\begin{problem}{\rm Find necessary and sufficient criteria under which RWRE in one dimension is strongly transient.}
\end{problem}
Our approach is based on the connection between ERWs and a class of critical branching processes (BPs) with random migration (see Section \ref{two} for details). It is close in spirit to the (second) Ray-Knight theorem (see, for example, \cite{T96}, where similar ideas were used for other types of self-interacting RWs). This approach was proposed for ERWs in \cite{BS08a} and, since then, seemed to dominate the study of one-dimensional ERWs under the (IID) assumption. The main benefits gained from this connection are: 
\begin{itemize}
\item[(i)] BPs associated to ERWs are markovian, while the original processes do not enjoy this property;
\item[(ii)] after rescaling, these BPs are well approximated by squared Bessel processes of generalized dimension.\vspace*{-1mm} 
\end{itemize}
From these diffusion approximations one can immediately conjecture such important properties of BPs as survival versus extinction, the tail asymptotics of the extinction time and of the total progeny (conditioned on extinction where appropriate). Rigorous proofs of these conjectures are somewhat technical, but, in a nutshell, they are based on standard martingale techniques applied to gambler-ruin-like problems. 

Diffusion approximations for BPs associated to ERWs and the mentioned above martingale techniques were used in \cite{KM11} to study the tail behavior of regeneration times of transient ERWs, which led to theorems about limit laws for these processes. In the current work we extend some of the results and techniques of \cite{KM11} and, in addition, apply the Doob transform to treat BPs conditioned on extinction. The results for conditioned BPs are then readily translated into the proof of Theorem \ref{ex}.

While the majority of results in the BPs literature rely on generating functions approach, diffusion approximations of BPs is also a well-developed subject. Its history goes back to \cite{Fe51} (see \cite[Chapter 9]{EK86} for precise statements and additional references). But it seems that diffusion approximations for our kind of BPs are not available in the literature. Moreover, among a wealth of results (obtained by any approach) about conditioned BPs we could not find those which would cover our needs (but see the related work \cite{Mel83} and the references therein).

We would like to point out one more aspect of the relationship between ERWs and BPs. At first (see, for example, \cite{BS08a}, \cite{KZ08}) there was a tendency to use known results for BPs to infer results about ERWs. Gradually, as we mentioned above, the study of ERWs required additional results about BPs, not covered by the literature. In \cite{KM11} all BP results needed for ERWs were obtained directly. In this work we continue the trend. Theorem \ref{sun} gives asymptotics of the tails of extinction time and the total progeny of a class of critical BPs with random migration and geometric offspring distribution conditioned on extinction. We believe that this result might be of independent interest and that our methods are sufficiently robust to be applicable to more general critical BPs with random migration.

Let us now describe how the present article is organized. We close the introduction  with some notation. In the next section we recall how excursions of ERWs are related to certain BPs. Section~\ref{moon} deals with diffusion approximations of these BPs. In Section~\ref{sunset} we prove that BPs conditioned on extinction can be approximated by the diffusions from Section~\ref{moon} conditioned on hitting zero. In Section~\ref{star} we use these results to obtain tail asymptotics of the extinction time and of the total progeny of BPs conditioned on extinction. Short Section~\ref{main} translates the obtained asymptotics into the proof of Theorem~\ref{ex}.
In the Appendix we collect and extend as necessary several auxiliary results from the literature, which we quote throughout the paper and which do not depend on the results from Sections \ref{moon}--\ref{main}.

{\bf Notation.} 
For any $I\subseteq [0,\infty)$  and $f:I\to\R$ we let 
$\si^f_y:=\inf\left\{t\in I:  f(t)\le y\right\}$ and
$\tau^f_y:=\inf\left\{t\in I:  f(t)\ge y\right\}$ be
the entrance time of $f$ into $(-\infty,y]$ and $[y,\infty)$, respectively. (Here $\inf\emptyset:=\infty$.) 
If $Z$ is a process with $P[\si_0^Z<\infty]>0$ then we denote by $\overline Z$ any process which has the same distribution as $Z$ under $P[\ \cdot\, \mid \si_0^{Z}<\infty]$. 
Whenever $X$ is a Markov process starting at time 0 we indicate the starting point $X(0)=x$ in expressions like $P_x[X\in A]$ by the subscript $x$ to $P$. 
The space of real-valued c\`adl\`ag functions on $[0,\infty)$ is denoted by $D[0,\infty)$ and convergence in distribution by $\Rightarrow$.
\section{Excursions of  RWs and  branching processes}\label{two}
We recall a relationship between  nearest neighbor paths from 1 to 0, representing RW excursions to the right, and BPs.  Among the first descriptions of this relation is \cite[Section 6]{Har52}. We refer to \cite[Sections 3, 4]{KZ08} and \cite[Section 2.1]{Pet} for detailed explanations in the  context of ERW.
 
Assume that the nearest neighbor random walk $(X_n)_{n\ge 0}$ starts at $X_0=1$, set $U_0:=1$ and let for $k\ge 1$, 
\begin{equation}\label{uk}
U_k:=\#\{n\ge 1:\ n<T_0,\ X_{n-1}=k,\ X_{n}=k+1\}
\end{equation}
be the number of upcrossings from $k$ to $k+1$ by the walk before time $T_0$.
If we set $\Delta_0^{(k)}:=0$ then  
\[\Delta_m^{(k)}:=\inf\left\{n\ge 1: \Delta_{m-1}^{(k)}<n\le T_0,\ X_{n-1}=k, X_n=k-1\right\},\qquad k,m\ge 1,\]
is, if finite, the time of the completion of the $m$-th downcrossing from $k$ to $k-1$ prior to $T_0$.
We define
\[\zeta_{m}^{(k)}:=\#\left\{n\ge 1: \Delta_{m-1}^{(k)}<n<\Delta_m^{(k)},\ 
X_{n-1}=k,\ X_n=k+1\right\},\qquad k,m\ge 1,
\]
to be the number of upcrossing from $k$ to $k+1$  between the $(m-1)$-th and the $m$-th downcrossing from  $k$ to $k-1$ before $T_0$. 
Then $U_{k+1}$ can be represented in BP form as
\[
U_{k+1}=\sum_{m=1}^{U_k}\zeta_{m}^{(k+1)}.
\]
Here $\zeta_{m}^{(k+1)}$ can be interpreted as the number of children of the $m$-th individual in the $k$-generation.
The joint distribution of these numbers depends on the RW model under consideration.
In the case of ERW 
it may be 
quite complicated, especially in the case where $T_0=\infty$ with positive probability. Therefore, we study instead of $U$ a slightly different BP $V$, the so-called \textit{forward BP} described in the following statement.  
\begin{prop}\label{coup}{\rm \bf (Coupling of ERW and forward BP)} Assume we are given $M\in\N$ and an ERW $X=(X_n)_{n\ge 0}$ which satisfies {\rm (IID), (WEL)} and {\rm (BD$_\mathrm{M}$)}. Then one may assume without loss of generality that there are on the same probability space 
$\N_0$-valued random variables $\xi_m^{(k)},\ m,k\ge 1,$ which define
a Markov chain $V=(V_k)_{k\ge 0}$ by $V_0:=1$ and 
\begin{equation}\label{sj}
V_{k+1}:=\sum_{m=1}^{V_k}\xi_{m}^{(k+1)},\qquad k\ge 0,
\end{equation}
such that under $P_1$ the following holds:
\begin{align}
&\parbox[h]{12cm}{ 
The random quantities $(\xi_1^{(k)},\ldots,\xi_M^{(k)}), \xi_{m}^{(k)}\ (m>M, k\ge 1)$ are independent.}\label{ind}\\
&\parbox[h]{12cm}{The random vectors $(\xi_1^{(k)},\ldots,\xi_M^{(k)})\ (k\ge 1)$ are identically distributed, $\N_0^M$-valued, vanish with positive probability,  and have a finite fourth moment.}\label{a3}\\
\label{ddM}
&\sum_{m=1}^M(\xi_m^{(1)}-1)\quad\mbox{has expected value $\delta$.}\\
&\label{gg}\parbox[h]{12cm}{
The random variables $\xi_m^{(k)}\ (m>M, k\ge 1)$ are geometrically distributed with parameter $1/2$ and expected value 1.}\\
&U_k\le V_k \qquad\mbox{for all $k\ge 0$ and }\label{jim}\\ \label{knopf}
&U=V\qquad\mbox{on the event $\{\si_0^U<\infty\}\cup\{\si_0^V<\infty\}$},
\end{align}
where $U$ is defined by {\rm (\ref{uk})}.
\end{prop}
Proposition \ref{coup} follows from the so-called  coin-toss construction of ERW described in \cite[Section 4]{KZ08}, see also \cite[Section 2]{Pet}.
Note that the above conditions (\ref{sj})--(\ref{gg}) do not completely characterize the distribution of $V$. For this 
statement (\ref{a3}) would have to be made stronger. However, we refrain from doing so, since the conditions (\ref{sj})--(\ref{gg}) are the only ones we need for our proofs to work. (The moment condition in (\ref{a3}) is inherited from the proof of \cite[Lemma 5.2]{KM11} and could be relaxed.) Indeed, we only make the following assumptions on $V$.\vspace*{2mm}

\noindent{\bf Assumptions on the offspring $\xi$ and the BP $V$.} For the remainder of the paper we assume that the Markov chain $V$ is defined by (\ref{sj}), where the offspring variables $\xi_m^{(k)},\ m,k\ge 1,$ satisfy (\ref{ind})--(\ref{gg}).

\begin{remark}\label{s1}{\rm {\bf (Cookies of strength 1 and BPs with immigration)}
In \cite[p.\ 1960]{KZ08} we describe how the above process $V$ can be viewed as a BP with migration, i.e.\ emigration and immigration.
If (IID) and (WEL) hold, but not necessarily (BD$_\mathrm{M}$), and if there is $\PP$-a.s.\ some random $K\in\N\cup\{\infty\}$ such that $\om(0,i)=1$ for all $1\le i<K$ and $\om(0,i)=1/2$ for all $i\ge K$ then one can couple the ERW in a way similar to the one described in Proposition \ref{coup} to a BP with immigration without emigration, see e.g.\ \cite[Section 3]{Bau13}. This kind of BP seems to be more tractable than BPs with immigration and emigration and several results are available in the BP literature which have direct implications for such ERWs. For example,  the recurrence/transience phase transition in $\delta$ can be obtained from \cite[Th.\ 1]{Pak71} or \cite[Th.\ 3]{Zub72}. For other examples see  Remarks \ref{simi} and \ref{we}.
}
\end{remark}

\begin{remark}{\rm {\bf (Other uses of the forward BP)}
The above mentioned relationship between excursions of RWs and BPs has been mainly used so far to translate result about BPs into results about RWs. The RW is then called the contour process associated to the BP.
We list a few examples.  Solomon's 
recurrence/transience theorem \cite[second part of Th.\ (1.7)]{So75} for RWRE follows from results by
Smith and Wilkinson \cite{SW69} about the extinction of Galton-Watson processes in random environment, see also \cite[Ch.\ VI.5]{AN72}. In \cite[p.\ 268]{Afa99} this relationship is shown to imply that for recurrent RWRE $P_1[T_0>n]\sim c/\log n$ as $n\to\infty$ for some constant $0<c<\infty$. In \cite[Th.\ 1]{KZ08} we used this correspondence and results from \cite{FYK90} for a proof of the recurrence/transience result about ERW mentioned above, see also Corollary \ref{rc2} below. 
In \cite{Bau13} and \cite{Bau} this connection is used to determine how many cookies (of maximal value $\om(x,i)=1$) are needed to change the recurrence/transience behavior of RWRE. And in \cite[Th.\ 1.7]{Pet} strict monotonicity with respect to the environment of the return probability of a transient ERW is shown to be inherited from monotonicity properties of this BP. 
}\end{remark}
\begin{remark}\label{d1}{\rm {\bf(Backward BP)} There is yet another family of branching processes associated to random walk paths, sometimes called the \textit{backward BPs}, see
\cite{BS08a}, \cite[Section 6]{KZ08}, \cite{KM11}, \cite[Th.\ 5.2]{KZ13}, and \cite[Section 2.2]{Pet}. 

We notice that all results of \cite{KM11} about backward BPs have the corresponding analogs for forward BPs, which are obtained by replacing $\delta$ (which is assumed to be positive in \cite{KM11}) with $1-\delta<1$ throughout. The proofs carry over essentially word for word without any additional changes. In what follows we simply quote such results. All additional results about forward BPs, in particular, for $\delta>1$, are supplied with detailed proofs or comments as appropriate. 
}
\end{remark}

\section{Diffusion approximation of unconditioned branching processes}\label{moon}
The main result of this section is Theorem \ref{uncond} about diffusion approximations of the process $V$. It extends \cite[Lemma~3.1]{KM11}, which only considered the process $V$ stopped at $\sigma^V_{\epsilon n}$ with $\epsilon>0$.

The limiting processes are defined in terms of solutions of the stochastic differential equation (SDE)
\begin{equation}\label{bes}
dY(t)=\delta\ dt+\sqrt{2Y^+(t)}\ dB(t),
\end{equation}
where $(B(t))_{t\ge 0}$ is a one-dimensional Brownian motion. For discussions of this particular SDE see e.g.\ \cite[Ch.\ V.48]{RW00} and \cite[Example IV-8.2]{IW89}. 
By \cite[Th.\  3.10, p.\ 299]{EK86} the SDE (\ref{bes})  has a weak solution   $Y=(Y(t))_{t\ge 0}$ for any initial distribution $\mu$ on $\R$ and any $\delta\in\R$. Due to the Yamada-Watanabe uniqueness theorem \cite[Th.\  1]{YW71} (see also \cite[Th.\  40.1]{RW00}) pathwise uniqueness holds for (\ref{bes}). By \cite[Prop.\  1]{YW71}
(see also \cite[Th.\  3.6, p.\ 296]{EK86}) distributional uniqueness holds as well. 
For $\delta\ge 0$, $2Y$ is a squared Bessel processes of dimension $2\delta$, see e.g.\ \cite[Chapter XI, \S 1]{RY99}. For $\delta<0$, $2Y$ coincides with squared Bessel processes of negative dimension (see \cite[Section 3]{GY03}) up to time $\si_0^Y$ and continues degenerately after time $\si_0^Y$ since by the strong Markov property a.s.\ 
\begin{equation}\label{dt}
Y(\si_0^Y+t)=\delta t\quad\mbox{for $t\ge 0, \delta<0$.}
\end{equation}

In order to obtain these diffusion approximations
we first introduce a modification $\widetilde V$ of the original process $V$ and state in Proposition \ref{500} a functional limit theorem for this process. The advantage of this process $\widetilde V$ is that it admits some nice martingales.
Note that (\ref{sj}) can be rewritten as
\[V_{k+1}=V_k+\sum_{m=1}^{V_k}(\xi_{m}^{(k+1)}-1).\]
This recursion is modified below in (\ref{vt}).
\begin{lemma}\label{mm} Let $x\in\Z$,
$\widetilde V_0:=x$ and let $\xi$ satisfy {\rm (\ref{ind})--(\ref{gg})}. Set
$v:={\rm Var}\left[\sum_{m=1}^M\xi_m^{(1)}\right].$
For  $k\in\N_0$ define
\begin{eqnarray}\label{vt}
\widetilde V_{k+1}&:=&\widetilde V_k+\sum_{m=1}^{\widetilde V_k\vee M}(\xi_m^{(k+1)}-1),\\
M_k&:=&\widetilde V_k-k\delta,\quad\mbox{and}\label{Mk}\\
A_k&:=&vk+2\sum_{m=0}^{k-1}(\widetilde V_m-M)^+.\label{Ak}
\end{eqnarray}
Then $(M_k)_{k\ge 0}$ and $(M_k^2-A_k)_{k\ge 0}$ are martingales with respect to the filtration $(\F_k)_{k\ge }$, where $\F_k$ is generated by $\xi_m^{(i)}, m\ge 1, 1\le i\le k$.
\end{lemma}
\begin{proof} By (\ref{a3})--(\ref{gg}), 
\begin{equation}\label{dd}
E\left[\sum_{m=1}^i(\xi_m^{(k+1)}-1)\right]=\delta\quad\mbox{for all $i\ge M$  and $k\ge 0$.}
\end{equation}
This implies the first statement.
To find the Doob decomposition of the submartingale $(M_k^2)_{k\ge 0}$ we compute
\begin{eqnarray*}
\lefteqn{E[M_{k+1}^2-M_k^2\mid\F_{k}]=E[(M_{k+1}-M_k+M_k)^2-M_k^2\mid\F_{k}]}\\
&=&E[(M_{k+1}-M_k)^2\mid\F_{k}]+2M_k E[M_{k+1}-M_k\mid\F_{k}]\\
& \stackrel{(\ref{Mk})}{=}& E[(\widetilde V_{k+1}-\widetilde V_k-\delta)^2\mid\F_k]\\
&\stackrel{(\ref{vt})}{=}&E\left[\left(\left(\sum_{m=1}^{\widetilde V_k\vee M}\left(\xi_m^{(k+1)}-1\right)\right)-\delta\right)^2\Bigg|\ \F_k\right]\ \stackrel{(\ref{dd})}{=}\ 
{\rm Var}\left[\sum_{m=1}^{\widetilde V_k\vee M}\left(\xi_m^{(k+1)}-1\right)\Bigg|\ \F_k\right]
\end{eqnarray*}
\begin{eqnarray*}
&\stackrel{(\ref{ind})}{=}&{\rm Var}\left[\sum_{m=1}^{M}\left(\xi_m^{(k+1)}-1\right)\right]+\sum_{m=M+1}^{\widetilde V_k\vee M}
{\rm Var}\left[\xi_{m}^{(k+1)}-1\right]\ =\
v+2\left(\widetilde V_k-M\right)^+.
\end{eqnarray*}
Recalling (\ref{Ak}) we obtain the second claim.
\end{proof}
\begin{prop}
\label{500}
Let $(x_n)_{n\ge 1}$ be a sequence of positive numbers which converges to $x>0$, let $\delta\in\R$ and assume that $\xi$ satisfies {\rm (\ref{ind})--(\ref{gg})}.
For each $n\in\N$ define $\widetilde V_n=(\widetilde V_{n,k})_{k\ge 0}$ and 
$\widetilde Y_n=(\widetilde Y_n(t))_{t\ge 0}$ by setting 
$\widetilde V_{n,0}:=\lfloor nx_n\rfloor$ and 
\begin{eqnarray*}
\widetilde V_{n,k+1}&:=&\widetilde V_{n,k}+\sum_{m=1}^{\widetilde V_{n,k}\vee M}(\xi_m^{(k+1)}-1),
\qquad \mbox{for $k\ge 0$ and} \\
\widetilde Y_n(t)&:=&\frac{\widetilde V_{n,\lfloor nt\rfloor}}{n}
\qquad \mbox{for $t\in[0,\infty)$.}
\end{eqnarray*}
Let $Y=(Y(t))_{t\ge 0}$ solve {\rm (\ref{bes})} with $Y(0)=x$. Then 
$\widetilde Y_n\stackrel{J_1}{\Longrightarrow} Y$ as $n\to\infty$.
\end{prop}
\begin{proof}
We are going to apply \cite[Th.\  4.1, p.\ 354]{EK86}.
To check the assumptions of this theorem we first let  $\mu$ be a  distribution on $\R$ and consider the $C_\R[0,\infty)$ martingale problem for $(A,\mu)$ with $A=\{(f,Gf): f\in C_c^\infty(\R)\}$, where $Gf:=(a/2)f''+\delta f'$ and $a(x):=2x^+$. This martingale problem  is well posed due to \cite[Cor.\  3.4, p.\ 295]{EK86} and our discussion above after  (\ref{bes}) concerning existence and distributional uniqueness of solutions of (\ref{bes}).

Now define for each $n\in\N$, $(M_{n,k})_{k\ge 0}$ and $(A_{n,k})_{k\ge 0}$ in terms of $\widetilde V_{n,k}$ as in (\ref{Mk}) and (\ref{Ak}), respectively.  For $t\in[0,\infty)$ set
\[
M_n(t):=\frac{M_{n,\lfloor tn\rfloor}}n,\quad
A_n(t):=\frac{A_{n,\lfloor tn\rfloor}}{n^2},\quad
B_n(t):=\frac{\lfloor tn\rfloor \delta}n.
\]
We are now going to check conditions (4.1)-(4.7) of \cite[Th.\ 4.1,
p.\ 354]{EK86}.  By Lemma \ref{mm}, $M_n-B_n$ and $M^2_n-A_n$ are
martingales for all $n\in\N$, i.e.\ conditions (4.1) and (4.2) are
satisfied.  To verify the remaining conditions (4.3)--(4.7) we fix  $r,T\in(0,\infty)$ and set $\tau_{n,r}:=\inf\{t>0:
|\widetilde Y_n(t)|\vee |\widetilde Y_n(t-)|\ge r\}$. To 
check condition (4.3), we have to show that 
\begin{equation}\label{fig}
\lim_{n\to\infty}
E\left[\sup_{t\le T\wedge\tau_{n,r}}|\widetilde Y_n(t)-\widetilde
  Y_n(t-)|^2\right]=0.
\end{equation}
This is a consequence of (\ref{gg}) and the fact that
the geometric distribution has exponential tails. More precisely, 
\begin{align*}
  E\bigg[&\sup_{t\le T\wedge\tau_{n,r}}|\widetilde Y_n(t)-\widetilde
    Y_n(t-)|^2\bigg]=\frac{1}{n^2}\,E\bigg[\max_{1\le k\le (Tn)\wedge
    \tau^{\widetilde V}_{\lfloor
      rn\rfloor}}\Big|\sum_{m=1}^{\widetilde V_{n,k-1}\vee
    M}(\xi^{(k)}_m-1)\Big|^2\bigg]\\
&\le \frac{2}{n^2}\,
  E\bigg[\max_{1\le k\le Tn}
  \Big|\sum_{m=1}^M(\xi^{(k)}_m-1)\Big|^2+\max_{1\le k\le (Tn)\wedge
    \tau^{\widetilde V}_{\lfloor
      rn\rfloor}}\Big|\sum_{m=M+1}^{\widetilde V_{n,k-1}}(\xi^{(k)}_m-1)\Big|^2\bigg]\\
&\le \frac{2T}{n}\,
  E\bigg[\Big|\sum_{m=1}^M(\xi^{(0)}_m-1)\Big|^2\bigg]+\frac{2}{n^2}\,
  E\bigg[\max_{1\le
    k\le Tn}\max_{M+1\le j\le rn}\Big|\sum_{m=M+1}^j(\xi^{(k)}_m-1)\Big|^2\bigg].
\end{align*}
The first term in the last line goes to $0$ as $n\to\infty$ and the
second term  is equal to 
\begin{eqnarray*}
  &&\frac{2}{n^2}\sum_{y\ge 0} P\bigg[\max_{1\le k\le
      Tn}\,\max_{M+1\le j\le rn}\Big|\sum_{m=M+1}^j(
      \xi_m^{(k)}-1)\Big|^2>y\bigg]\\ 
&\le&\frac{2r^{3/2}}{n^{1/2}}+\frac{2}{n^2}
  \sum_{y>(rn)^{3/2}} P\bigg[\max_{1\le k\le
      Tn}\,\max_{M+1\le j\le rn}\Big|\sum_{m=M+1}^j(
      \xi_m^{(k)}-1)\Big|^2>y\bigg].
\end{eqnarray*}
The first term in the last line vanishes as $n\to\infty$. 
Applying the union bound and
Lemma~\ref{gld} to the last probability we find that the second term
does not exceed
\begin{equation*}
  4rT\sum_{y>(rn)^{3/2}} e^{-y/(6(rn\vee\sqrt{y}))} \le 4rT
  \sum_{y>(rn)^{3/2}} e^{-y^{1/3}/6}\to
  0\quad\text{as }\ n\to\infty.
\end{equation*}
This finishes the proof of (\ref{fig}).
Conditions (4.4) and (4.6) of \cite[Th.\ 4.1,
p.\ 354]{EK86} (with $b\equiv
\delta$) hold obviously. Condition (4.5) is fulfilled, since
\[\sup_{t\le T\wedge\tau_{n,r}}|A_n(t)-A_n(t-)|\le \frac{v+2(nr+M)}{n^2}. \]
For (4.7) we consider for all $t\le T\wedge\tau_{n,r}$,
\begin{eqnarray*}
  \lefteqn{\left|A_n(t)-2\int_0^t \widetilde Y^+_n(s)\ ds\right|}\\
  &=&\left|\frac{v\lfloor tn\rfloor}{n^2}
    +\frac 2{n^2}\sum_{m=0}^{\lfloor tn\rfloor-1}(\widetilde
    V_{n,m}-M)^+-\frac 2n\int_0
    ^{\frac{\lfloor tn\rfloor}{n}} \widetilde V^+_{n,\lfloor
      sn\rfloor}\ 
    ds-2\int_{\frac{\lfloor tn\rfloor}n}^t Y^+_n(s)\ ds\right|\\
  &\le&\frac{vt}n+\frac 2{n^2}\sum_{m=0}^{\lfloor
    tn\rfloor-1} \left|(\widetilde V_{n,m}-M)^+-\widetilde
    V_{n,m}^+\right|+ \frac 2n \sup_{s<t}Y^+_n(s)\ \le\
  \frac{(v+2M)T+ 2r}n,
\end{eqnarray*}
which does not depend on $t$ any more and converges to 0 as
$n\to\infty$. Thus, (4.7) holds as well.  The theorem
follows now from \cite[Th.\ 4.1, p.\ 354]{EK86}.
\end{proof}
To be able to apply the continuous mapping theorem to Proposition \ref{500} we need the following statement.  Define for every $f\in D[0,\infty)$ and $y\in\R$ by
\[\varphi_y(f):=f(\cdot\wedge \si_y^f)\]
the function $f$ stopped after entering $(-\infty,y]$.
\begin{lemma}\label{cont}
Let $\delta\in\R$, $0<\eps<x<\infty$ and let $\psi$ be any of the following three mappings defined on $D[0,\infty)$:
\[f\mapsto\si_\eps^f\in [0,\infty],\quad
 f\mapsto\varphi_\eps(f)
\in D[0,\infty),\quad
f\mapsto\int_0^{\si_\eps^f}f^+(s)\ ds\in[0,\infty].
\]
Denote by
${\rm Cont}(\psi):=\left\{f\in D[0,\infty):  \psi\mbox{ is continuous at $f$}\right\}$ the set of continuity points of $\psi$.
Then the solution $Y$ of {\rm (\ref{bes})} 
satisfies $P_x[Y\in {\rm Cont}(\psi)]=1$.
\end{lemma}
\begin{proof} For $0<\eps<x<\infty$ let
\begin{eqnarray*}
F&:=&\left\{f\in C[0,\infty)\ \Big|\ f(0)=x,\quad \si_\eps^f<\infty\Rightarrow \mbox{$f$ has no local minimum at $\si_\eps^f$}\right\}.
\end{eqnarray*}
Then under the conditions of the lemma
$P_x\left[Y\in F\right]=1.$
Indeed, it follows from the strong Markov property and \cite[Lemma (46.1) (i)]{RW00} that $Y$ a.s.\ does not have a local minimum at $\si_\eps^Y$. 

Consequently, it suffices to show that  $F\subseteq {\rm Cont}(\psi)$. For $\psi=\si_\eps^\cdot$ and $\psi=\varphi_\eps$ this follows from 
\cite[Ch.\ VI, Prop.\ 2.11]{JS87} and \cite[Ch.\ VI, Prop.\ 2.12]{JS87}, respectively. (In the notation of \cite[Ch.\ VI, 2.9]{JS87}, $\si_\eps^f=S_a(\al)$ with $\al:=e^{-f}$ and $a:=e^{-\eps}$. Note that $f\mapsto e^{-f}$ is continuous w.r.t.\ the $J_1$-topology.)

For 
$\psi(f)=\int_0^{\si_\eps^f}f^+(s)\ ds$, we assume that $D[0,\infty)\ni f_n\stackrel{J_1}{\longrightarrow}f\in F$. We need to show that  $\psi(f_n)\to\psi(f)$. Since $\si_\eps^\cdot$ is continuous, as shown above, $\si_\eps^{f_n}\to\si_\eps^f$. If $\si_\eps^f<\infty$ then $\si_\eps^{f_n}<\si_\eps^f+1=:T$ for $n$ large, and hence 
$\left|\psi(f_n)-\psi(f)\right|\leq T\sup_{t\in[0,T]}|f_n(t)-f(t)|,$
which converges to 0 as $n\to\infty$, see e.g.\ \cite[Ch.\ VI, Prop.\ 1.17b]{JS87}. If $\si_\eps^f=\infty$, then for any $T<\infty$, $\si_\eps^{f_n}\ge T$ for $n$ large and thus
\[\psi(f_n)\ge \int_0^Tf_n^+(s)\ ds\begin{array}[t]{c}\longrightarrow\vspace*{-2mm}\\{\scriptstyle n\to\infty}\end{array}\int_0^Tf^+(s)\ ds\begin{array}[t]{c}\longrightarrow\vspace*{-2mm}\\{\scriptstyle T\to\infty}\end{array}\int_0^\infty f^+(s)\ ds=\infty\]
since $f(s)>\eps$ for all $s\ge 0$.
\end{proof}
\begin{theorem}\label{uncond} {\bf (Convergence of unconditioned processes)} Let $(x_n)_{n\ge 1}$ be a sequence of positive numbers which converges to $x>0$, let $\delta\in\R\backslash\{1\}$, and assume {\rm (\ref{ind})--(\ref{gg})}.
For each $n\in\N$ define $V_n=(V_{n,k})_{k\ge 0}$ and $Y_n=(Y_n(t))_{t\ge 0}$ by setting $V_{n,0}:=\lfloor nx_n\rfloor$ and 
\begin{eqnarray*}
V_{n,k+1}&:=&\sum_{m=1}^{V_{n,k}}\xi_{m}^{(k+1)}\qquad\mbox{for $k\ge 0$ and}\\
Y_n(t)&:=&\frac{V_{n,\lfloor nt\rfloor}}{n}\qquad\mbox{for $t\in[0,\infty)$}.
\end{eqnarray*}
Let $Y$ be a solution of {\rm (\ref{bes})} with $Y(0)=x$.
Then 
 $Y_n\stackrel{J_1}{\Longrightarrow}Y\left(\cdot\wedge \si_0^Y\right)$  as $n\to\infty$.
\end{theorem}
\begin{proof}
Let $\widetilde V_n$ and $\widetilde Y_n$ be defined as in Proposition \ref{500}, where $\widetilde V_{n,0}=V_{n,0}$ for all $n$.
We denote by $d^\circ_\infty$  the $J_1$-metric on $D[0,\infty)$ as defined in \cite[(16.4)]{Bil99}.

We first consider the case $\delta>1$. In this case $P_x[\si_0^Y=\infty]=1$ and, hence, $Y=Y\left(\cdot\wedge \si_0^Y\right)$, see e.g.\ \cite[(48.5)]{RW00}.
Moreover, on the event  $\{\si_M^{V_n}=\infty\}$ we have $\widetilde V_n=V_n$ and thus
$\widetilde Y_n=Y_n$.
Consequently, we have for every $\eps>0$, $P\left[d^\circ_\infty\left(\widetilde Y_n,Y_n\right)>\eps\right] \le
P\left[\si_{M}^{V_n}<\infty\right]\to 0$ as $n\to\infty$ due to Corollary \ref{rc}.
Consequently, $d^\circ_\infty\left(\widetilde Y_n,Y_n\right)$ converges in distribution to 0 as $n\to\infty$. Therefore, by Proposition \ref{500} and \cite[Th.\  3.1]{Bil99}  (a  ``convergence together'' theorem), $Y_n\stackrel{J_1}{\Longrightarrow} Y$ as $n\to\infty$. 
This completes the proof in case $\delta>1$. 

Now we consider the case $\delta<1$.
Our first goal is to show that
\begin{equation}\label{j1}
\varphi_0\left(\widetilde Y_n\right)\stackrel{J_1}{\Longrightarrow} \varphi_0(Y)\qquad \mbox{as $n\to\infty$}.
\end{equation}
We aim to use \cite[Th.\ 3.2]{Bil99}, quoted as Lemma \ref{billy} in the Appendix, for this purpose. First observe that for all $m\in\N,$
$\varphi_{1/m}(\widetilde Y_n)\stackrel{J_1}{\Longrightarrow} \varphi_{1/m}(Y)$ as $n\to\infty$
due to Proposition \ref{500}, Lemma \ref{cont} and the continuous mapping theorem. Moreover, $\varphi_{1/m}(Y)\Longrightarrow \varphi_0(Y)$ as $m\to\infty$ since $Y$ has a.s.\ continuous paths. For the proof of (\ref{j1}) it therefore suffices to show, due to Lemma \ref{billy}, that 
\begin{equation}\label{c}
\lim_{m\to\infty}\limsup_{n\to\infty}P\left[d^\circ_\infty\left(
\varphi_{1/m}\left(\widetilde Y_n\right),
\varphi_{0}\left(\widetilde Y_n\right)\right)>2\eps\right]=0\quad\mbox{for every $\eps>0$}.
\end{equation}
For the proof of (\ref{c}) we use \cite[(12.16)]{Bil99} and see that for all $\eps>0 
$, $n\in\N$, and $y\in(0,x\wedge \eps)$,
\begin{eqnarray}\nonumber
\lefteqn{\hspace*{-8mm}
P\left[d^\circ_\infty\left(
\varphi_{y}\left(\widetilde Y_n\right),
\varphi_{0}\left(\widetilde Y_n\right)\right)>2\eps\right]\ \le\ P\left[\left\|
\varphi_{y}\left(\widetilde Y_n\right)-
\varphi_{0}\left(\widetilde Y_n\right)\right\|_\infty>2\eps\right]}\\
\qquad &\le&
P\left[\si_0^{\widetilde V_n}<\infty, \widetilde V_{n,\si_0^{\widetilde V_n}}<-n\eps\right]+P\left[\sup\left\{\widetilde Y_n(s):  \si_{y}^{\widetilde Y_n}\le s\le \si_{0}^{\widetilde Y_n}\right\}>\eps\right]. \label{jojo}
\end{eqnarray}
The first term in (\ref{jojo}) is 0 for large enough $n$ since $\widetilde V_{\si_0^{\widetilde V}}\ge -M$.  The second term is $\le
P_{\lfloor yn\rfloor}\left[\tau_{\eps n}^{\widetilde
    V}<\si_0^{\widetilde V}\right]$.  Lemma \ref{5.5ab} now yields
(\ref{j1}) if we choose $y=1/m$.  If we choose $y=M/n$ then the above
estimate and Lemma \ref{5.5ab} give that
\begin{equation}\label{abc}
d^\circ_\infty\left(\varphi_{M/n}\left(\widetilde Y_n\right), \varphi_{0}\left(\widetilde Y_n\right)\right)\to 0\quad \mbox{in distribution as $n\to\infty$.}
\end{equation}
Consequently, by (\ref{j1}) and \cite[Th.\  3.1]{Bil99},
$\varphi_{M/n}\left(\widetilde Y_n\right)
\stackrel{J_1}{\Longrightarrow} \varphi_0(Y)$ as $n\to\infty$.
However, recall that $\widetilde V_{n,k}=V_{n,k}$ for all $0\le k\le \si_{M}^{V_n}$ and therefore
$\varphi_{M/n}\left(\widetilde Y_n\right)=\varphi_{M/n}\left(Y_n\right)$.
Hence,
$\varphi_{M/n}\left(Y_n\right)\stackrel{J_1}{\Longrightarrow} \varphi_0(Y)$ as $n\to\infty$.
As in (\ref{abc}),
$d^\circ_\infty(\varphi_{M/n}(Y_n), \varphi_{0}(Y_n))$
tends to 0 in distribution as $n\to\infty$. 
The claim  for $\delta<1$  now follows from  another application of \cite[Th.\  3.1]{Bil99}. (Note that $\varphi_0\left(Y_n\right)=Y_n$ since 0 is absorbing for $V$.)
\end{proof}

\section{Diffusion approximation of conditioned branching processes}\label{sunset}
The main result of this section is the following. Recall that $\hV$ is obtained from $V$ by conditioning on $\{\si_0^V<\infty\}$. In particular,  by Corollary \ref{rc}, $\hV=V$ if $\delta<1$.
\begin{theorem}\label{cond} {\bf (Convergence of conditioned processes)} 
Assume the conditions of Theorem \ref{uncond} and 
let $\hY=(\hY(t))_{t\ge 0}$ be a solution to 
\begin{equation}\label{bes2}
d\hY(t)=(1-|\delta-1|)\ dt+\sqrt{2\hY^+(t)}\ dB_t,\qquad \hY(0)=x.
\end{equation}
Then as $n\to\infty$, 
\begin{eqnarray}
\hY_n&\stackrel{J_1}\Longrightarrow&\hY\left(\cdot\wedge \si_0^{\hY}\right),\label{cconv1}\\
\si_0^{\hY_n}&\Longrightarrow&\si_0^{\hY},\quad\mbox{and}\label{cconv2}\\
\int_0^{\si_0^{\hY_n}}\hY_n(s)\ ds&\Longrightarrow&\int_0^{\si_0^{\hY}}\hY(s)\ ds.\label{cconv3}
\end{eqnarray}
\end{theorem}
In the proof the (harmonic) function $h$ defined by
\begin{equation}\label{h}
h(n):=P_n[\si_0^V<\infty],\qquad n\in\N_0.
\end{equation}
will play an important role. Note that it follows from (\ref{sj})  that $h(x)$ is non-increasing in $x$. 
\begin{remark}\label{vc}{\rm {\bf (Doob transform)} 
Recall that $\hV$ is Doob's $h$-transform of $V$ with $h$ as defined in (\ref{h}),
see e.g.\ \cite[Chapter 7.6.1]{LPW09}.  By this we mean that $\hV$ is a Markov chain with transition probabilities $P_x[\hV_n=y]=P_x[V_n=y]\frac{h(y)}{h(x)}$.
More generally, it follows from the strong Markov property, that for any stopping time $\si\le \si_0^V$ and all $x,y\in\N_0$, 
\begin{equation}\label{ht}
P_x[\hV_\si=y]=\frac{P_x[V_\si=y, \si_0^V<\infty]}{P_x[\si_0^V<\infty]}=P_x[V_\si=y]\frac{h(y)}{h(x)}.
\end{equation} 
In many cases, a Doob transform of a process belongs to the same class of processes as the process itself. For example, the asymmetric simple RW on $\Z$ with probability $p\in(1/2,1)$ of stepping to the right, start at 1 and absorption at 0 is, conditioned on hitting 0, an asymmetric simple RW on $\Z$ with probability $p$ of stepping to the left. (In this case $h(x)=((1-p)/p)^x$ for $x\ge 0$.)
Similarly, a supercritical Galton-Watson process conditioned on extinction is a subcritical Galton-Watson process, see e.g.\ \cite[Ch.\ I.12, Th.\ 3]{AN72}. In a wider sense, squared Bessel processes of dimension $d>2$ conditioned on hitting 0 are squared Bessel processes of dimension 
$4-d$.

If similarly $\overline X$, i.e.\ $X$ conditioned on hitting 0, were under $P_1$ an ERW satisfying (IID), (WEL) and (BD$_{\overline M}$) for some $\overline M$, or, equivalently, if $\hV$ were of the form described in Proposition \ref{coup} then Theorem \ref{cond} would follow from Theorem \ref{uncond}. However, we do not expect that the conditioned processes $\overline X$ and $\hV$ are of this form on the microscopic level. Theorem \ref{cond} shows that nevertheless on a macroscopic scale $\hV$ does behave as $V$ with drift parameter $\overline\delta=1-|\delta-1|$.   }
\end{remark}
First we investigate the tail behavior of $h$. 
\begin{prop}{\bf (Asymptotics of $h$)}\label{har}
Let $\delta>1$. Then there is $\con{hh}\in(0,\infty)$ such that
\begin{equation}\label{hard}
\lim_{n\to\infty}{n^{\delta-1}}h(n)=c_{\ref{hh}}.
\end{equation}
\end{prop}
\begin{remark}\label{simi}{\rm
In the special case described in Remark \ref{s1}
formula (\ref{hard}) also follows from \cite[Th.\ 4]{Pak72}, 
see also the discussion in \cite[pp.\ 921,922]{Hop85}. }
\end{remark}
\begin{proof}[Proof of Proposition \ref{har}] 
The proof is similar to that of \cite[Lemma 8.1]{KM11}.
Let $\delta>1$.  
First we show that it suffices to prove that 
\begin{equation}\label{ga}
g(a):=\lim_{m\to\infty}a^{m(\delta-1)} h(\lfloor a^m\rfloor)\in(0,\infty)\quad\mbox{for all $a\in(1,2]$.}
\end{equation}
Let $a\in(1,2]$ and denote $m_n:=\lfloor\log_an\rfloor$ for $n\in\N$. Then, by monotonicity of $h$, 
\[a^{m_n(\delta-1)}h\left(a^{m_n+1}\right)\le n^{\delta-1}h(n)\le 
a^{(m_n+1)(\delta-1)}h\left(a^{m_n}\right)
\]
for all $n\in\N$ and hence, by (\ref{ga}),
\[a^{1-\delta}g(a)\le\liminf_{n\to\infty}n^{\delta-1}h(n)\le 
\limsup_{n\to\infty}n^{\delta-1}h(n)\le a^{\delta-1} g(a).\]
Since $0<g(a)<\infty$ this implies
\[1\le \frac{\limsup_{n\to\infty}n^{\delta-1}h(n)}{\liminf_{n\to\infty}n^{\delta-1}h(n)}
\le a^{2(\delta-1)}.\]
Letting $a\searrow 1$ proves the claim of the theorem.

It remains to show (\ref{ga}). Fix $a\in(1,2]$ and choose $\ell_0$ according to Lemma \ref{5.3}. Then for all $m> \ell\ge \ell_0$,
\begin{eqnarray*}
\lefteqn{a^{m(\delta-1)} h(\lfloor a^m\rfloor)=
a^{m(\delta-1)}P_{\lfloor a^m\rfloor}[\si_{a^\ell}\le \si_0<\infty]}
\\
&\ge&a^{m(\delta-1)}P_{\lfloor a^m\rfloor}[\si_{a^\ell}<\infty]P_{\lfloor a^\ell\rfloor}[\si_0<\infty]\\
&\stackrel{(\ref{ip})}{\ge}&  a^{(m-\ell)(\delta-1)}\,h_\ell^-(m)\, a^{\ell(\delta-1)}\,h\left(\lfloor a^\ell\rfloor\right)\
\stackrel{(\ref{pp})}{\ge}\ K_1(\ell)\ a^{\ell(\delta-1)}h\left(\lfloor a^\ell\rfloor\right)>0.
\end{eqnarray*}
Hence, since $K_1(\ell)\to 1$,
\[\liminf_{m\to\infty}a^{m(\delta-1)} h(\lfloor a^m\rfloor)\ge
\limsup_{\ell\to\infty}a^{\ell(\delta-1)} h(\lfloor a^\ell\rfloor)>0,\]
which establishes the existence of $g(a)>0$.
To rule out $g(a)=\infty$ observe that
\begin{eqnarray*}
a^{m(\delta-1)} h(\lfloor a^m\rfloor)&\le& a^{m(\delta-1)} P_{\lfloor a^m\rfloor}\left[\si_{a^{\ell_0}}<\infty\right]\stackrel{(\ref{ip})}{\le}
a^{(m-\ell_0)(\delta-1)}\ h_{\ell_0}^+(m)\ a^{\ell_0(\delta-1)}\\
&\stackrel{(\ref{pp})}{\le}& K_2(\ell_0)\ a^{\ell_0(\delta-1)}<\infty
\end{eqnarray*}
for all $m>\ell_0$.
\end{proof}
\begin{lemma}\label{6.1}
Let $\delta\in\R\backslash\{1\}$. Then there is $\con{c4}\in(0,1)$ such that for all $k,x\in\N$ and $n\ge 0$,
\begin{equation}\label{bar}
P_n\left[\#\left\{i\in\left\{1,\ldots,\si_0^{\hV}\right\}: \hV_i\in[x,2x)\right\}>2xk\right]\
\le\
P_n\left[\rho_0<\si_0^{\hV}\right]c_{\ref{c4}}^k,
\end{equation}
where $\rho_0:=\inf\{i\ge 0: \hV_i\in[x,2x)\}$.
\end{lemma}
\begin{proof} For $\delta<1$, where $\hV=V$, this is the statement of \cite[Prop.\ 6.1]{KM11}.
For the case $\delta>1$ we slightly modify the proof of \cite[Prop.\ 6.1]{KM11} as follows. First we show that there is $\con{ci}>0$ such that for all $x\in\N$,
\begin{equation}
\min_{x\le z<2x}P_{z}\left[\si_{x/2}^{\hV}<x\right]>c_{\ref{ci}}.\label{i}
\end{equation}
By the strong Markov property and
monotonicity with respect to the starting point we have for all $2M\le x\le z<2x$,
\begin{eqnarray*}\nonumber
P_{z}\left[\si_{x/2}^{\hV}<x\right]&=&\frac{P_{z}[\si_{x/2}^{V}<x,\ \si_0^V<\infty]}{
h(z)}
\ \ge\ \frac{h(\lceil x/2\rceil)}{
h(z)}\, P_{z}[\si_{x/2}^{V}<x]\\
& \ge& P_{z}[\si_{x/2}^{V}<x]\ \ge\  P_{2x}[\si_{x/2}^{V}<x]\ =\ P_{2x}[\si_{x/2}^{\widetilde V}<x].
\end{eqnarray*}
The last expression  is strictly positive for all $x>0$ and 
converges due to Proposition \ref{500} and Lemma \ref{cont} as $x\to\infty$ to $P_2[\si_{1/2}^Y<1]>0$, where $Y$ solves the SDE (\ref{bes}). This proves (\ref{i}). 

Next we show that there is  $\con{cii}>0$ such that  for all $x\in\N$,
\begin{equation}
\min_{0\le z\le x/2}P_{z}\left[\si_0^{\hV}<\tau_x^{\hV}\right]>c_{\ref{cii}}.\label{ii}
\end{equation}
For the proof of (\ref{ii}) note that by the strong Markov property and
monotonicity of $h$ for all $0\le z\le x/2, z\in\N_0$,
\begin{eqnarray*}
P_{z}\left[\tau_x^{\hV}<\si_0^{\hV}\right]&=&
\frac{P_{z}[\tau_x^{V}<\si_0^{V}<\infty]}{h(z)}\ \le\
\frac{h(x)}{h(z)}P_{z}[\tau_x^{V}<\si_0^{V}]
\\
 &\le&
\frac{h(x)}{h(z)}\ \le\
\frac{h(x)}{h(\lfloor x/2\rfloor)}
\ \le\
2^{1-\delta}\frac{x^{\delta-1}h(x)}{\lfloor x/2\rfloor^{\delta-1}h(\lfloor x/2\rfloor)},
\end{eqnarray*} 
which converges due to Proposition \ref{har} to $2^{1-\delta}<1$ as $x\to\infty$. Since the left hand side of (\ref{ii}) is strictly positive for all $x$ this implies (\ref{ii}).

Now define $\rho_i:=\inf\{n>\rho_{i-1}+2x: \hV_n\in[x,2x)\}$ for all $i\in\N$. Then the left hand side of (\ref{bar}) is 
less than or equal to
\begin{equation}\label{drs}
P_n\left[\rho_k<\si_0^{\hV}\right]=
P_n\left[\rho_{k}<\si_0^{\hV}\ \big|\ \rho_{k-1}<\si_0^{\hV}\right]P_n\left[\rho_{k-1}<\si_0^{\hV}\right].
\end{equation}
By the strong Markov property for $\hV$,
\begin{eqnarray*}
\lefteqn{P_n\left[\rho_{k}<\si_0^{\hV}\ \big|\ \rho_{k-1}<\si_0^{\hV}\right]\le\max_{x\le z<2x}P_z\left[\rho_1<\si_0^{\hV}\right]}\\
&=&\max_{x\le z<2x}\left(P_z\left[\rho_1<\si_0^{\hV}, \si_{x/2}^{\hV}<x\right]+P_z\left[\rho_1<\si_0^{\hV}, \si_{x/2}^{\hV}\ge x\right]\right)\\
&\le &\max_{x\le z<2x}\left(P_z\left[\rho_1<\si_0^{\hV}\mid \si_{x/2}^{\hV}<x\right]
P_z\left[\si_{x/2}^{\hV}<x\right]+1-P_z\left[\si_{x/2}^{\hV}< x\right]\right)\\
&= &\max_{x\le z<2x}\left(1-P_z\left[\si_{x/2}^{\hV}<x\right]\left(1-P_z\left[\rho_1<\si_0^{\hV}\mid \si_{x/2}^{\hV}<x\right]\right)\right)\\
&= &1-\min_{x\le z<2x}\left(P_z\left[\si_{x/2}^{\hV}<x\right]P_z\left[\rho_1>\si_0^{\hV}\mid \si_{x/2}^{\hV}<x\right]\right)\\
&\le &1-\left(\min_{x\le z<2x}P_z\left[\si_{x/2}^{\hV}<x\right]\right)\left(\min_{0\le y\le x/2}P_y\left[\rho_0>\si_0^{\hV}\right]\right)\ \le 1-c_{\ref{ci}}c_{\ref{cii}}
\end{eqnarray*}
by   (\ref{i}) and (\ref{ii}). Substituting this into (\ref{drs}) and iterating gives the claim with $c_{\ref{c4}}:=1-c_{\ref{ci}}c_{\ref{cii}}$.
\end{proof}
The next lemma states that Lemma \ref{5.5ab} also holds for $\hV$. 
\begin{lemma}\label{5.5c} Let $\delta>1$. 
For every $\ga>0$ there is $\con{eps0}(\ga)\in(0,\infty)$ such that 
for all $n\in\N$,
$P_n\left[\tau_{c_{\ref{eps0}}n}^{\hV}<\si_0^{\hV}\right]<\ga$.
\end{lemma}
\begin{proof}
By the strong Markov property and monotonicity of $h$, for all $t>1$,
\[\limsup_{n\to\infty}P_n\left[\tau_{tn}^{\hV}<\si_0^{\hV}\right]=\limsup_{n\to\infty}\frac{P_n\left[\tau_{tn}^{V}<\si_0^{V}<\infty\right]}{h(n)}\le
\limsup_{n\to\infty}\frac{h(\lceil tn\rceil)}{h(n)}= t^{1-\delta}\]
due to Proposition \ref{har}. Since $t^{1-\delta}\to 0$ as $t\to\infty$ and $P_n[\si_0^{\hV}<\infty]=1$ for all $n\in\N$ by definition of $\hV$
this finishes the proof.
\end{proof}
\begin{lemma}\label{7.1} Let $\delta\in\R\backslash\{1\}$.
For each $\eps>0$ there is $\con{ee}(\eps)$ such that $P_n[\si_0^{\hV}>c_{\ref{ee}}(\eps)n]<\eps$ for all $n\in\N$.
\end{lemma}
\begin{proof} The proof is the same as the one of \cite[Prop. 7.1]{KM11}. It uses Lemmas \ref{6.1} and \ref{5.5c} instead of Proposition~6.1 and (5.5) of \cite{KM11} respectively.  
\end{proof}
\begin{proof}[Proof of Theorem \ref{cond}]
We first prove (\ref{cconv1}).
For $\delta<1$, (\ref{cconv1}) follows from Theorem \ref{uncond} since $\hV=V$.

Now assume $\delta>1$.
By \cite[Ch.\ 4, Th.\ 2.12]{EK86}
it suffices to show that the one-dimensional marginal distributions converge, i.e.\ that for all $t>0$,
\begin{equation}\label{1st}
\hY_n\left(t\right)\Longrightarrow \hY\left(t\wedge \si_0^{\hY}\right)\qquad\mbox{as $n\to\infty$}.
\end{equation} 
(This theorem is applicable since the semigroup corresponding to $\hY(\cdot\wedge\si_0^{\hY})$ is Feller on the space of continuous functions on $[0,\infty)$ vanishing at infinity, see
\cite[Ch.\ 8, Th.\ 1.1, Cor.\ 1.2]{EK86}.)

For the proof of  (\ref{1st}) let us assume for the moment that we have already shown that for all
$t,\eps>0$,
\begin{equation}\label{ep}
\hY_n\left(t\wedge\si_\eps^{\hY_n}\right)\Longrightarrow \hY\left(t\wedge \si_\eps^{\hY}\right)\qquad\mbox{as $n\to\infty$}.
\end{equation} 
Since $\hY$ has continuous trajectories we have for all $t>0$,
$\hY\left(t\wedge \si_\eps^{\hY}\right)
\Rightarrow
\hY\left(t\wedge \si_0^{\hY}\right)$ as $\eps\searrow 0$.
Moreover, for all $\eta>0$,
\[
\lim_{\eps\searrow 0}\limsup_{n\to\infty}P\left[\left|\hY_n\left(t\wedge\si_\eps^{\hY_n}\right)-\hY_n(t)\right|>\eta\right]\le
\lim_{\eps\searrow 0}\limsup_{n\to\infty}
P_{\lfloor n\eps\rfloor}\left[\tau_{n\eta}^{\hV}<\si_0^{\hV}\right]
=0
\]
due to Lemma \ref{5.5c}. Therefore, we can apply Lemma \ref{billy} and obtain (\ref{1st}).

It remains to verify (\ref{ep}). Fix. $t,\eps>0$.
We need to show that for any bounded and continuous function $f:\R\to\R$, 
\begin{equation}\label{im}
\lim_{n\to\infty}
E\left[f\left(\hY_n\left(t\wedge\si_\eps^{\hY_n}\right)\right)\right]= E\left[f\left(\hY\left(t\wedge \si_\eps^{\hY}\right)\right)\right].
\end{equation}
By (\ref{ht}),
\[
E\left[f\left(\hY_n\left(t\wedge\si_\eps^{\hY_n}\right)\right)\right]=
E\left[f\left(Y_n\left(t\wedge\si_\eps^{Y_n}\right)\right)\frac{h\left(nY_n\left(t\wedge\si_\eps^{Y_n}\right)\right)}{h(\lfloor xn\rfloor)}\right]= A_n+B_n,
\]
where
\begin{eqnarray*}
A_n&:=&E\left[f\left(Y_n\left(t\wedge\si_\eps^{Y_n}\right)\right)\left(\frac{h\left(nY_n\left(t\wedge\si_\eps^{Y_n}\right)\right)}{h(\lfloor xn\rfloor)}-\left(\frac{x}{Y_n\left(t\wedge\si_\eps^{Y_n}\right)\vee\eps/2}\right)^{\delta-1}\right)\right],\\
B_n&:=&x^{\delta-1}E\left[f\left(Y_n\left(t\wedge\si_\eps^{Y_n}\right)\right)\left(Y_n\left(t\wedge\si_\eps^{Y_n}\right)\vee\eps/2\right)^{1-\delta}\right].
\end{eqnarray*}
Note that by Theorem \ref{uncond}, Lemma \ref{cont}, and the continuous mapping theorem, $Y_n(t\wedge\si_\eps^{Y_n})$ converges in distribution as $n\to\infty$ to 
$\varphi_\eps(t)
$, where $\varphi_\eps:=Y(\cdot\wedge\si_\eps^{Y})$ and $Y$ solves (\ref{bes}) with initial condition $Y(0)=x$.
Hence, since the function $g(x):=f(x)(x\vee\eps/2)^{1-\delta}$ is bounded and continuous,
\begin{eqnarray}\nonumber
\lim_{n\to\infty}B_n&=& x^{\delta-1}E\left[f\left(\varphi_\eps(t)\right)\left(\varphi_\eps(t)\vee\eps/2\right)^{1-\delta}\right]
\ =\ x^{\delta-1}E\left[f\left(\varphi_\eps(t)\right)\left(\varphi_\eps(t)\right)^{1-\delta}\right]\nonumber
\\
 &=& E\left[f\left(\hY\left(t\wedge\si_\eps^{\hY}\right)\right)\right],
\label{vhs}
\end{eqnarray}
where the last identity follows from a change of measure as follows. Since $Y$ solves (\ref{bes})  with $Y(0)=x$, $\varphi_\eps$ 
solves 
\begin{equation}\label{bos}
d\varphi_\eps(s)=\delta\I_{\varphi_\eps(s)>\eps}\, ds+\sqrt{2\varphi_\eps(s)}\I_{\varphi_\eps(s)>\eps}\, dB(s), \quad \varphi_\eps(0)=x,
\end{equation}
see e.g.\ \cite[Prop.\ II-1.1 (iv)]{IW89}.
By It\^o's formula, 
using that $\varphi_\eps\ge \eps$,
\begin{align*}
  d\ln \varphi_\eps(s)&=\frac1{\varphi_\eps(s)}\,d\varphi_\eps(s)-\frac12\,\frac{1}{\varphi_\eps^2(s)}\, \left(d\varphi_\eps(s)\right)^2
  \\&\stackrel{(\ref{bos})}{=}\frac{\I_{\varphi_\eps(s)>\eps}}{\varphi_\eps(s)}\left(\delta\,ds+\sqrt{2\varphi_\eps(s)})\,dB(s)\right)
  -\frac{\I_{\varphi_\eps(s)>\eps}}{\varphi_\eps(s)}\,ds \\&=
\frac{\delta-1}{\varphi_\eps(s)}\I_{\varphi_\eps(s)>\eps}\,ds+\sqrt{\frac{2}{\varphi_\eps(s)}}\I_{\varphi_\eps(s)>\eps}\,dB(s).
\end{align*}
Therefore,
\begin{align*}
Z_\eps(t)&:=\exp\left((1-\delta)\left(\int_0^{t}\frac{\delta-1}{\varphi_\eps(s)}\I_{\varphi_\eps(s)>\eps}\,ds
+\int_0^{t}\sqrt{\frac{2}{\varphi_\eps(s)}}\I_{\varphi_\eps(s)>\eps}\,dB(s)\right)\right)\\
&=\exp\left((1-\delta)\left(\ln \varphi_\eps(t)-\ln
  x\right)\right)\ =\ x^{\delta-1}\left(\varphi_\eps(t)\right)^{1-\delta}.
\end{align*}
Now consider the measure $\widetilde P_\eps$ with $d\widetilde P_\eps/d P=Z_\eps(t)$. 
By Girsanov's transformation, see e.g.\
\cite[Th.\ IV-4.2]{IW89} with $\al(t,x):=\sqrt{2x}\I_{x>\eps}$, $\beta(t,x):=\delta\I_{x>\eps}$ and $\ga(t,x):=(1-\delta)\sqrt{2/x}\I_{x>\eps}$, the process $(\varphi_\eps(s))_{0\le s\le t}$ satisfies
\[d\varphi_\eps(s)=(2-\delta)\I_{\varphi_\eps(s)>\eps}\ ds +\sqrt{2\varphi_\eps(s)}\I_{\varphi_\eps(s)>\eps}\, d\widetilde B(s),\quad \varphi_\eps(0)=x,\]
where $\widetilde B$ is a standard Brownian motion w.r.t.\ $\widetilde P_\eps$. Hence $\varphi_\eps(t)$ has under $\widetilde P$ the same distribution as $\hY(t\wedge \si_\eps^{\hY})$ under $P$. This implies  (\ref{vhs}).

For the proof of (\ref{im}) it remains to show that $A_n\to 0$ as $n\to\infty$. Let $\kappa:=\|f\|_\infty$ and abbreviate $\varphi_{n,\eps}:=Y_n\left(t\wedge\si_\eps^{Y_n}\right)$. Then 
\begin{eqnarray}\nonumber
|A_n|&\le&\kappa E\left[\left|\frac{h\left(n\varphi_{n,\eps}\right)}{h(\lfloor xn\rfloor)}-\left(\frac{x}{\varphi_{n,\eps}\vee\eps/2}\right)^{\delta-1}\right|\right],\\
&=&\kappa x^{\delta-1}E\left[\left|\frac{h\left(n\varphi_{n,\eps}\right)\left(\left({\varphi_{n,\eps}\vee\eps/2}\right)n\right)^{\delta-1}-(xn)^{\delta-1}h(\lfloor xn\rfloor)}{(xn)^{\delta-1}h(\lfloor xn\rfloor)\left(\varphi_{n,\eps}\vee\eps/2\right)^{\delta-1}}\right|\right],\nonumber\\
&\le&\kappa \left(\frac{2x}{\eps}\right)^{\delta-1}E\left[\left|\frac{h\left(n\varphi_{n,\eps}\right)\left(\left({\varphi_{n,\eps}\vee\eps/2}\right)n\right)^{\delta-1}-(xn)^{\delta-1}h(\lfloor xn\rfloor)}{(xn)^{\delta-1}h(\lfloor xn\rfloor)}\right|\right].\label{wurst}
\end{eqnarray}
Now let $\ga\in(0,c_{\ref{hh}}/2)$, where $c_{\ref{hh}}$ is the constant from Proposition \ref{har}. By Proposition  \ref{har} there is $n_0\in\N$ such that 
\begin{equation}\label{wag}
|s^{\delta-1}h(\lfloor s\rfloor)-c_{\ref{hh}}|<\ga\quad\mbox{ for all $s\in[n_0,\infty)$.}
\end{equation} 
Thus for  large enough $n$ 
the right-hand side of (\ref{wurst}) is less than or equal to
\begin{eqnarray*}
&&\frac\kappa{c_{\ref{hh}}-\ga} \left(\frac{2x}{\eps}\right)^{\delta-1}E\left[\left|h\left(n\varphi_{n,\eps}\right)\left(\left({\varphi_{n,\eps}\vee\eps/2}\right)n\right)^{\delta-1}-(xn)^{\delta-1}h(\lfloor xn\rfloor)\right|\right]\\
&\le&\frac{2\kappa}{c_{\ref{hh}}} \left(\frac{2x}{\eps}\right)^{\delta-1} \left(E\left[\left|h\left(n\varphi_{n,\eps}\right)\left(\left({\varphi_{n,\eps}\vee\eps/2}\right)n\right)^{\delta-1}-c_{\ref{hh}}\right|\right]+\ga\right)\\
&=&\frac{2\kappa}{c_{\ref{hh}}} \left(\frac{2x}{\eps}\right)^{\delta-1}
\bigg(
E\left[\left|h\left(n\varphi_{n,\eps}\right)\left({\varphi_{n,\eps}}n\right)^{\delta-1}-c_{\ref{hh}}\right|,\varphi_{n,\eps}\ge \eps/2\right]\nonumber\\
&&\hspace*{26mm}E\left[\left|h\left(n\varphi_{n,\eps}\right)\left(\eps n/2\right)^{\delta-1}-c_{\ref{hh}}\right|,\varphi_{n,\eps}< \eps/2\right]+\ga\bigg)\\
&\le&\frac{2\kappa}{c_{\ref{hh}}} \left(\frac{2x}{\eps}\right)^{\delta-1}
\left(\ga+\left(\left(\eps n/2\right)^{\delta-1}+c_{\ref{hh}}\right)P\left[\varphi_{n,\eps}< \eps/2\right]+\ga\right)\ \le\
\frac{6\ga \kappa}{c_{\ref{hh}}} \left(\frac{2x}{\eps}\right)^{\delta-1}
\end{eqnarray*}
Here we used 
in the last step that   $P\left[\varphi_{n,\eps}< \eps/2\right]$ decays exponentially fast as $n\to\infty$ due to (\ref{under}).
Letting $\ga\searrow 0$ yields $\lim_{n}A_n=0$. This proves (\ref{cconv1}).

For the next statement, (\ref{cconv2}), we shall use  Lemma \ref{billy}. It follows from 
(\ref{cconv1}), Lemma \ref{cont}, the continuous mapping theorem, and the a.s.\ continuity of $\hY$ that $\si_{1/k}^{\hY_n}\begin{array}[t]{c}\Longrightarrow\vspace*{-3mm}\\{\scriptstyle n}\end{array}\si_{1/k}^{\hY}\begin{array}[t]{c}\Longrightarrow\vspace*{-3mm}\\{\scriptstyle k}\end{array}\si_0^{\hY}$.
To verify the condition corresponding to (\ref{jon})
fix $\eps>0$ and let $c_{\ref{ee}}(\cdot)$ be the function of Lemma \ref{7.1}. Choose $g:\N\to (0,\infty)$ such that $c_{\ref{ee}}(g(k))\le \eps k$ for all $k\in\N$ and $g(k)\to 0$ as $k\to\infty$. Then for all $k,n\in\N$,
\begin{eqnarray}\label{jaja}
P\left[\si_0^{\hY_n}-\si_{1/k}^{\hY_n}>\eps\right]&\leq& P_{\lfloor \frac nk\rfloor}\left[\si_0^{\hV}>\eps n\right]\\
&\leq& P_{\lfloor \frac nk\rfloor}\left[\si_0^{\hV}>c_{\ref{ee}}(g(k))\left\lfloor \frac nk\right\rfloor\right]\ \le\ g(k),\nonumber
\end{eqnarray}
due to Lemma \ref{7.1}. Letting first $n\to\infty$ and then $k\to\infty$ 
implies the condition corresponding to (\ref{jon}). Consequently, Lemma \ref{billy} shows (\ref{cconv2}).

The final statement, (\ref{cconv3}), is obtained similarly. As above, it follows from 
(\ref{cconv1}), Lemma \ref{cont}, the continuous mapping theorem, and the a.s.\ continuity of $Y$ that
\[\int_0^{\si_{1/k}^{\hY_n}}\hY_n(s)\ ds\begin{array}[t]{c}\Longrightarrow\vspace*{-3mm}\\{\scriptstyle n\to\infty}\end{array}
\int_0^{\si_{1/k}^{\hY}}\hY(s)\ ds
\begin{array}[t]{c}\Longrightarrow\vspace*{-3mm}\\{\scriptstyle k\to\infty}\end{array}
\int_0^{\si_0^{\hY}}\hY(s)\ ds.
\]
Moreover, for all $\eps\in (0,1)$,
\begin{eqnarray*}
\lefteqn{P\left[\int_{0}^{\si_0^{\hY_n}}\hY_n(s)\ ds-\int_{0}^{\si_{1/k}^{\hY_n}}\hY_n(s)\ ds>\eps\right]
\ =\ P\left[\int_{\si_{1/k}^{\hY_n}}^{\si_0^{\hY_n}}\hY_n(s)\ ds>\eps\right]}\\
&\le& P\left[\sup\left\{\hY_n(s):  \si_{1/k}^{\hY_n}\le s\le \si_{0}^{\hY_n}\right\}\left(\si_{0}^{\hY_n}-
\si_{1/k}^{\hY_n}\right)>\eps\right]\\
&\le&P\left[\sup\left\{\hY_n(s):  \si_{1/k}^{\hY_n}\le s\le \si_{0}^{\hY_n}\right\}>\eps\right]+P\left[\si_{0}^{\hY_n}-
\si_{1/k}^{\hY_n}>\eps\right]\\
&\le& P_{\lfloor \frac nk\rfloor}\left[\tau_{\eps n}^{\hV}<\si_0^{\hV}\right]+g(k)
\end{eqnarray*}
as in (\ref{jojo}) and due to (\ref{jaja}). 
Letting first $n\to\infty$ and then $k\to\infty$ this converges to 0 due to Lemma \ref{5.5c} and the choice of $g$. Having verified the assumptions of Lemma \ref{billy} statement (\ref{cconv3})  follows from this lemma.
\end{proof}
\section{Asymptotics of diffusions and branching processes}\label{star}
In this section we derive certain asymptotics of the conditioned BP $\hV$ from those of the approximating squared Bessel process as stated in Lemma \ref{sca}. Our goal is the following result:
\begin{theorem}\label{sun}
Let  $V=(V_n)_{n\ge 0}$ be defined as in {\rm (\ref{sj})} such that {\rm (\ref{ind})--(\ref{gg})} holds with $V_0=1$ and let $\delta\in\R\backslash\{1\}$.
Then there are $\con{v3c},\con{v5c}\in(0,\infty)$ such that 
\begin{eqnarray}\label{wer}
\lim_{n\to\infty}n^{|\delta-1|}\ P_1[\si_0^{\hV}>n]&=&c_{\ref{v3c}},\\
\lim_{n\to\infty}n^{|\delta-1|/2}\ P_1\left[\sum_{i=0}^{\si_0^{\hV}}\hV_i>n\right]&=&c_{\ref{v5c}}.
\label{wie}
\end{eqnarray}
\end{theorem}
\begin{remark}\label{we}{\rm In the special case described in Remark \ref{s1}
formula (\ref{wer}) follows for $0<\delta<1$ from the tail behavior of the extinction time of recurrent critical BPs with immigration described in \cite[Th.\  2, first part of (21)]{Zub72}. In this case it is for $\delta>1$ reminiscent of the unproven claims made in \cite[middle of page 225]{IS85} about the extinction time of transient critical BPs with immigration.
More general results concerning BPs with migration applying to the case $\delta\in(-\infty,1)\backslash\{0,-1\}$ are given in \cite[Theorems 1,3,4]{FYK90}, see also \cite[Th.\ A (iv)]{KZ08}.}
\end{remark}
\begin{lemma}\label{C}
Let $\delta\in\R\backslash\{1\}$. Then there is $\con{c5}\in(0,\infty)$ such that\\ $\lim_{n\to\infty}n^{|\delta-1|}P_1\left[\tau_n^{\hV}<\si_0^{\hV}\right]=c_{\ref{c5}}$.
\end{lemma}
\begin{proof} For $\delta<1$ the statement follows from  \cite[(C), Lem.\ 8.1]{KM11}.

Now let $\delta>1$. Then by the strong Markov property and monotonicity of $h$,
\[\limsup_{n\to\infty}n^{\delta-1}P_1\left[\tau_n^{\hV}<\si_0^{\hV}\right]
\le \limsup_{n\to\infty}n^{\delta-1}\frac{P_1\left[\tau_n^V<\infty\right]h(n)}{h(1)}=\frac{1-h(1)}{h(1)}c_{\ref{hh}}=:c_{\ref{c5}}\]
due to Proposition \ref{har}.
On the other hand, for all $\ga>1$ by the strong Markov property and monotonicity of $h$,
\begin{eqnarray*}
\lefteqn{\liminf_{n\to\infty}n^{\delta-1}P_1\left[\tau_n^{\hV}<\si_0^{\hV}\right]
\ge \liminf_{n\to\infty}n^{\delta-1}\frac{P_1\left[\tau_n^{V}<\si_0^{V}<\infty, V_{\tau_n^V}\le \ga n\right]}{h(1)}}\\
&\ge & \liminf_{n\to\infty}n^{\delta-1}\frac{P_1\left[\tau_n^{V}<\infty, V_{\tau_n^V}\le \ga n\right]h(\lfloor\ga n\rfloor)}{h(1)}\ =\ \frac{1-h(1)}{h(1)}c_{\ref{hh}}\ga^{1-\delta}\begin{array}[t]{c}\longrightarrow\vspace*{-3mm}\\{\scriptstyle \ga\searrow 1}\end{array}c_{\ref{c5}},
\end{eqnarray*}
where we used in the last identity (\ref{over}) and Proposition \ref{har}.
\end{proof}
\begin{lemma}\label{P6.2}
Let $\delta\in\R\backslash\{1\}$. Then for every $\eps>0$,
\[\lim_{x\searrow 0}\limsup_{n\to\infty}\, n^{|\delta-1|}\, P_1\left[\#\left\{i\in\{1,\ldots,\si_0^{\hV}\}: \hV_i< xn\right\}>\eps n\right]=0.\]
\end{lemma}
\begin{proof}
The proof is the same as the one of \cite[Prop. 6.2]{KM11}. It uses Lemmas \ref{6.1} and \ref{C} instead of Proposition~6.1 and (5.4) of \cite{KM11}, respectively. 
\end{proof}
\begin{proof}[Proof of Theorem \ref{sun}.] 
We first prove (\ref{wer}) with $c_{\ref{v3c}}:=c_{\ref{c5}}a(1-|\delta-1|)$, where $c_{\ref{c5}}$ is as in Lemma \ref{C} and the function $a$ is as in Lemma \ref{3.35}.
The proof generalizes the one of \cite[Th.\ 2.1]{KM11}, which covers the case $\delta<1$. Let $A_n:=\{\si_0^{\hV}>n\}$ for $n\ge 0$.

\textit{Lower bound in} (\ref{wer}).
Fix $x>0$. Then by the strong Markov property for $\hV$ we have for all $n\ge 0$,
\begin{eqnarray}
n^{|\delta-1|}P_1\left[A_n\right] &\ge& n^{|\delta-1|}P_1\left[\si_0^{\hV}>n+\tau_{xn}^{\hV},\hV_{\tau_{xn}^{\hV}}\le 2 xn\right]\nonumber\\
&=&n^{|\delta-1|}E_1\left[P_{\hV_{\tau_{xn}^{\hV}}}[A_n],\ \si_0^{\hV}>\tau_{xn}^{\hV}, \hV_{\tau_{xn}^{\hV}}\le 2 xn\right]\label{bb}\\
&\ge &n^{|\delta-1|}P_1\left[\si_0^{\hV}>\tau_{xn}^{\hV},\hV_{\tau_{xn}^{\hV}}\le 2 xn\right]\inf_{y\in[x,2 x]}P_{\lfloor  yn\rfloor}\left[
A_n\right].\label{bb2}
\end{eqnarray}
Now we 
choose suitable $y_n\in[x,2 x]$ where the infimum in (\ref{bb2}) is attained. Then the expression in (\ref{bb2}) can be estimated from below by
\[
\Big((xn)^{|\delta-1|}P_1\left[\si_0^{\hV}>\tau_{xn}^{\hV}\right]
\Big) x^{-|\delta-1|}P_{\lfloor  y_nn\rfloor}\left[A_n\right]\\
-\ n^{|\delta-1|}P_1\left[\tau_{xn}^{\hV}<\infty, \hV_{\tau_{xn}^{\hV}}> 2 xn\right].
\]
 The second term above vanishes as $n\to\infty$ due to (\ref{cos}). 
Therefore, by Lemma \ref{C}, 
 for all $x>0$,
\begin{eqnarray}
\liminf_{n\to\infty}n^{|\delta-1|}P_1\left[A_n\right]
&\ge&\nonumber
c_{\ref{c5}}x^{-|\delta-1|}\liminf_{n\to\infty}P_{\lfloor  y_nn\rfloor}\left[A_n\right]\\
&=&
c_{\ref{c5}}x^{-|\delta-1|}\lim_{k\to\infty}P_{\lfloor  y_{n_k}n_k\rfloor}\left[A_{n_k}\right]\label{kuh}
\end{eqnarray}
for some increasing sequence $(n_k)_{k\ge 0}$. Choose a 
 subsequence $(m_k)_{k\ge 0}$ of  $(n_k)_{k\ge 0}$ along which $(y_{m_k})_{k\ge 0}$ converges to
some $y\in[x,2x]$.  
Then  by (\ref{cconv2}), the limit in (\ref{kuh}) is  equal to
 $P_y[\si_0^{\hY}>1]$
where $\hY$ solves the SDE (\ref{bes2}).
By scaling (Lemma \ref{sca}), 
\[P_y[\si_0^{\hY}>1]=P_1[\si_0^{\hY}>1/y]\ge P_1[\si_0^{\hY}>1/x]\sim x^{|\delta-1|}a(1-|\delta-1|)
\] as $x\searrow 0$ by  (\ref{L3.3}).
This proves the lower bound in (\ref{wer}).

\textit{Upper bound in} (\ref{wer}). Fix $\eps \in (0,1)$ and estimate for all  $x>0$ and $n\ge 0$,
\begin{eqnarray}n^{|\delta-1|}P_1\left[A_n\right]
&\le& n^{|\delta-1|}P_1\left[A_n, \tau_{xn}^{\hV}\le \eps n\right]+n^{|\delta-1|}P_1\left[\tau_{xn}^{\hV}\wedge \si_0^{\hV}>\eps n\right]\label{wo0}\\
&\le&n^{|\delta-1|}
P_1\left[A_n, \tau_{xn}^{\hV}\le \eps n,\ \hV_{\tau_{xn}^{\hV}}\le (1+\eps )x n\right]\label{wo1}\\
&&+\ n^{|\delta-1|}P_1\left[\tau_{xn}^{\hV}<\infty, \hV_{\tau_{xn}^{\hV}}> (1+\eps )x n\right]\label{wo2}\\
&&+\ n^{|\delta-1|}P_1\left[\#\left\{i\in\{1,\ldots,\si_0^{\hV}\}: \hV_i< xn\right\}>\eps n\right].
\label{wo3}
\end{eqnarray}
The expression in (\ref{wo2}) vanishes  as $n\to\infty$ due to  (\ref{cos}). 
The term in (\ref{wo3}) vanishes as well  due to Lemma \ref{P6.2}
if we let first $n\to\infty$ and then $x\searrow 0$. 
For the treatment of (\ref{wo1}) let $B_n:=\left\{\si_0^{\hV}>(1-\eps )n\right\}$ for $n\ge 0$. Then  by the strong Markov property the quantity in (\ref{wo1}) is less than or equal to
\begin{eqnarray}\label{nana}
&&n^{|\delta-1|}
E_1\left[P_{\tau_{xn}^{\hV}}[B_n], \tau_{xn}^{\hV}<\si_0^{\hV},\ \hV_{\tau_{xn}^{\hV}}\le (1+\eps )x n\right]\\
&\le&
n^{|\delta-1|}P_1\left[\tau_{xn}^{\hV}<\si_0^{\hV}\right]\ \sup_{y\in[x,(1+\eps )x]}P_{\lfloor yn\rfloor}\left[B_n\right].\nonumber
\end{eqnarray}
By choosing suitable $y_n\in[x,(1+\eps ) x]$ where the supremum in (\ref{nana}) is attained the expression in (\ref{nana}) can be written as
\[\Big((xn)^{|\delta-1|}P_1\left[\tau_{xn}^{\hV}<\si_0^{\hV}\right]\Big) x^{-|\delta-1|}P_{\lfloor y_{n}(x)\,n\rfloor}\left[B_n\right].\]
As above the first factor converges to $c_{\ref{c5}}$ as $n\to\infty$ by Lemma \ref{C}. 
Summarizing we get that for all $\eps \in(0,1),$
\begin{eqnarray}\limsup_{n\to\infty}n^{|\delta-1|}P_1\left[A_n\right]
&\le&c_{\ref{c5}}\limsup_{x\searrow 0}\, x^{-|\delta-1|}\limsup_{n\to\infty}\,  P_{\lfloor y_{n}(x)\,n\rfloor}\left[B_n\right]\nonumber\\
&=&c_{\ref{c5}}\limsup_{x\searrow 0}\, x^{-|\delta-1|}\lim_{k\to\infty}\,  P_{\lfloor y_{n_{k,x}}(x)\,n_{k,x}\rfloor}\left[B_{n_{k,x}}\right]\label{glu}
\end{eqnarray}
for suitable increasing sequences $(n_{k,x})_{k\ge 0}$, $x>0$. Choose for all $x>0$ a subsequence $(m_{k,x})_{k\ge 0}$ of $(n_{k,x})_{k\ge 0}$ along which  $(y_{m_{k,x}}(x))_{k\ge 0}$ converges to some $y(x)\in[x,(1+\eps )x]$ and let $\hY$ solve the SDE (\ref{bes2}). Then by (\ref{cconv2}), the expression in (\ref{glu}) is equal to 
\begin{eqnarray*}
&&c_{\ref{c5}}\limsup_{x\searrow 0}\, x^{-|\delta-1|}
P_{y(x)}\left[\si_0^{\hY}>1-\eps \right]
 \stackrel{\rm L.\ \ref{sca}}{=}
c_{\ref{c5}}\limsup_{x\searrow 0}x^{-|\delta-1|}P_{1}\left[\si_0^{\hY}>\frac{1-\eps }{y(x)}\right]\\
&\le& c_{\ref{c5}}\limsup_{x\searrow 0}x^{-|\delta-1|}P_{1}\left[\si_0^{\hY}>\frac{1-\eps }{(1+\eps )x}\right]
\ \stackrel{(\ref{L3.3})}{=}\ c_{\ref{v3c}}\left(\frac{1+\eps }{1-\eps }\right)^{|\delta-1|}
\begin{array}[t]{c}\longrightarrow\vspace*{-2mm}\\{\scriptstyle \eps\searrow 0}\end{array}
c_{\ref{v3c}}.
\end{eqnarray*}
This completes the proof of (\ref{wer}).\vspace*{2mm}

Now we turn to the proof of (\ref{wie}) with $c_{\ref{v5c}}:=c_{\ref{c5}}b(1-|\delta-1|)$, where $c_{\ref{c5}}$ is as in Lemma \ref{C} and the function $b$ is as in Lemma \ref{3.35}.
The proof is similar to the one of \cite[Lemma 4.1]{KM11}.
It is enough to consider convergence along the subsequence $(n^2)_{n\ge 0}$. 
Moreover, it suffices to show 
\begin{equation}\label{xn}
\lim_{x\searrow 0}\lim_{n\to\infty}n^{|\delta-1|}\ P_1\left[A_n,\ \tau_{xn}^{\hV}<\si_0^{\hV} \right]=c_{\ref{v5c}},\quad\mbox{where}\ A_n:=\Bigg\{\sum_{i=0}^{\si_0^{\hV}}\hV_i>n^2\Bigg\}.
\end{equation}
Indeed, for all $x>0$ due to (\ref{wer})
\begin{eqnarray*}
n^{|\delta-1|}\ P_1\left[A_n,\ \si_0^{\hV}<\tau_{xn}^{\hV} \right]\le
n^{|\delta-1|}\ P_1\left[xn\si_0^{\hV}>n^2\right]
\begin{array}[t]{c}\longrightarrow\vspace*{-2mm}\\{\scriptstyle n\to\infty}\end{array}
c_{\ref{v3c}}x^{|\delta-1|}
\begin{array}[t]{c}\longrightarrow\vspace*{-2mm}\\{\scriptstyle x\searrow 0}\end{array}0.
\end{eqnarray*}

\textit{Lower bound in} (\ref{xn}). 
Fix $x>0$. Then 
\begin{equation}\label{term}
n^{|\delta-1|}\ P_1\left[A_n,\ \tau_{{xn}}^{\hV}<\si_0^{\hV} \right]
\ge n^{|\delta-1|}\ P_1\left[\sum_{i=\tau_{{xn}}^{\hV}}^{\si_0^{\hV}}{\hV}_i>n^2,\ \tau_{{xn}}^{\hV}<\si_0^{\hV},\ \hV_{\tau_{{xn}}^{\hV}}\le2 xn \right].
\end{equation}
This can be estimated from below by the expression in (\ref{bb}). 
By the same argument as after (\ref{bb}), using (\ref{cconv3}) instead of (\ref{cconv2}), we obtain
for a suitable $y\in[x,2 x]$ and a solution $\hY$ of (\ref{bes2}),
\begin{eqnarray*}\lefteqn{\liminf_{n\to\infty}n^{|\delta-1|}\ P_1\left[A_n,\ \tau_{{xn}}^{\hV}<\si_0^{\hV} \right]\ge c_{\ref{c5}}x^{-|\delta-1|}P_y\left[\int_0^{\si_0^{\hY}}\hY(s)\ ds>1\right]}\\
&\stackrel{\rm L.\ \ref{sca}}{=}&c_{\ref{c5}}x^{-|\delta-1|}P_1\left[\int_0^{\si_0^{\hY}}\hY(s)\ ds>\frac 1{y^2}\right]\ \ge\ 
c_{\ref{c5}}x^{-|\delta-1|}P_1\left[\int_0^{\si_0^{\hY}}\hY(s)\ ds>\frac 1{x^2}\right]\\
&\begin{array}{c}{\rm \scriptstyle L.\ \ref{3.35}}\vspace*{-2mm}\\ \longrightarrow\vspace*{-2mm}\\{\scriptstyle x\searrow 0}\end{array}&
c_{\ref{c5}}b(1-|\delta-1|)
\ =\ c_{\ref{v5c}}.
\end{eqnarray*}

\textit{Upper bound in} (\ref{xn}). We estimate the term in (\ref{term})  from above for all $n\ge 0$ and $x,\eps\in(0,1)$ by
\begin{eqnarray}
\lefteqn{\hspace*{-7mm}n^{|\delta-1|}\left(P_1\left[\sum_{i=\tau_{{xn}}^{\hV}}^{\si_0^{\hV}}\hV_i>(1-\eps x)n^2,\ \tau_{{xn}}^{\hV}<\si_0^{\hV}\right]\nonumber
+ P_1\left[\sum_{i=0}^{\tau_{xn}^{\hV}-1}\hV_i>\eps xn^2,\ \tau_{{xn}}^{\hV}<\si_0^{\hV}\right]\right)}\\
&\le&n^{|\delta-1|}P_1\left[\sum_{i=\tau_{{xn}}^{\hV}}^{\si_0^{\hV}}\hV_i>(1-\eps )n^2,\ \tau_{{xn}}^{\hV}<\si_0^{\hV}, V_{\tau_{xn}^{\hV}}\le  (1+\eps )x n \right]\label{weg}\\
&& +\  n^{|\delta-1|}P_1\left[\tau_{xn}^{\hV}<\infty, \hV_{\tau_{xn}^{\hV}}> (1+\eps )x n\right]+n^{|\delta-1|}P_1\left[\eps n\le\tau_{xn}^{\hV}<\si_0^{\hV}\right].\label{sum}
\end{eqnarray}
As above, see (\ref{wo2}) and the second term on the right-hand side of (\ref{wo0}), the sum in (\ref{sum}) is negligible.
The expression in (\ref{weg}) can be estimated from above by the strong Markov property by
the expression in (\ref{nana}), 
where $B_n:=\left\{\sum_{i=0}^{\si_0^{\hV}}\hV_i>(1-\eps )n^2\right\}$. 
By the same argument as after (\ref{nana}), using (\ref{cconv3}) instead of (\ref{cconv2}), we obtain for suitable $y(x)\in[x,(1+\eps )x]$ and $\hY$ a solution of (\ref{bes2}),
\begin{eqnarray*}
\lefteqn{\limsup_{x\searrow 0}\limsup_{n\to\infty}n^{|\delta-1|}\ P_1\left[A_n,\ \tau_{xn}^{\hV}<\si_0^{\hV} \right]}\\
&\le& c_{\ref{c5}}\limsup_{x\searrow 0}x^{-|\delta-1|}P_{y(x)}\left[\int_0^{\si_0^{\hY}}\hY(s)\ ds>1-\eps \right]\\
&\stackrel{\rm L.\ \ref{sca}}{\le}&
 c_{\ref{c5}}\limsup_{x\searrow 0}x^{-|\delta-1|}P_1\left[\int_0^{\si_0^{\hY}}\hY(s)\ ds>\frac{1-\eps }{(1+\eps )^2x^2}\right]\\
&\stackrel{\rm  L.\ \ref{3.35}}{=}&
c_{\ref{c5}}b(1-|\delta-1|)\left(\frac{1+\eps }{\sqrt{1-\eps }}\right)^{|\delta-1|}\ 
\longrightarrow
c_{\ref{v5c}}\qquad\mbox{as $\eps \searrow 0$.}
\end{eqnarray*}
\end{proof}
\section{Proof of the main result}\label{main}
\begin{proof}[Proof of Theorem \ref{ex}]
Observe that in the setting of Proposition \ref{coup}, $\si_0^U=\si_0^V$ due to (\ref{knopf}) and therefore, by definition (\ref{uk}),
\begin{equation}\label{tt}
\sup\{X_n:\ n<T_0\}=\si_0^V\quad\mbox{and}\quad T_0\I_{T_0<\infty}=\left(-1+2
\sum_{i=0}^{\si_0^V}V_i\right)\I_{\si_0^V<\infty}.
\end{equation}
(See e.g.\ \cite[Th.\ 5]{Har52} for the second identity.)
Claims (\ref{M0}) and (\ref{T0}) now follow from Theorem \ref{sun}.

Now consider the case $\delta=1$ and let $0<\eps<1$. Let $\hat \PP$ be another probability measure which satisfies (IID), (WEL), and (BD$_\mathrm{M}$)  and $\hat\PP[\om(0,i)\le t]\ge \PP[\om(0,i)\le t]$ for all $i\ge 1$ and $t\in[0,1]$ and whose corresponding parameter $\hat\delta$, defined as in (\ref{del}), satisfies $1-\eps<\hat\delta<1$. Then by monotonicity as stated in 
\cite[Lemma 15]{Zer05} we obtain $\hat\EE[P_{1,\om}[T_n<T_0]]\le P_1[T_n<T_0]$ and $\hat\EE[P_{1,\om}[T_0>n]]\le P_1[T_0>n]$. (Note that the additional condition of a.s.\ positivity of $2\om(x,i)-1$ which was assumed in \cite{Zer05} is not needed for the proof of \cite[Lemma 15]{Zer05}.) Applying (\ref{M0}) and (\ref{T0}) to $\hat \PP$ then yields the first two statements in (\ref{3s}). 
   
As for the remaining claims about the tail of $R$ 
consider the first step of the walk and use
the fact that $P_{x,\om}[n\le T_0<\infty]$ does not depend on $\om(0,\cdot)$ to obtain that
\begin{eqnarray}\nonumber
\lefteqn{P_0[n<R<\infty]\ =\ \EE\left[P_{0,\om}[n<R<\infty]\right]}\\
& =& \EE\left[ \om(0,1)P_{1,\om}[n\le T_0<\infty]\right]\label{ben}
+\EE\left[(1-\om(0,1))P_{-1,\om}[n\le T_0<\infty]\right]\\
& =& \EE[ \om(0,1)]P_{1}[n\le T_0<\infty]\nonumber
+\EE[1-\om(0,1)]P_{-1}[n\le T_0<\infty]
\end{eqnarray}
due to (IID). Because of (WEL) both $\EE[ \om(0,1)]>0$ and $\EE[1-\om(0,1)]>0$. Moreover,
$P_{-1}[n\le T_0<\infty]=\widetilde P_{1}[n\le T_0<\infty]$, where $\widetilde P_1[\cdot]=\widetilde \EE[P_{1,\om}[\cdot]]$ and $\om$ is distributed under $\widetilde \PP$ like $1-\om$ under $\PP$. The parameter $\widetilde\delta$ for $\widetilde \PP$ is equal to $-\delta$. Therefore, (\ref{T0}) implies
$n^{(|\delta-1|\wedge|-\delta-1|)/2}\, P_0[n<R<\infty]\to c_{\ref{du}}$ as $n\to\infty$
i.e.\ (\ref{tu}) since $|\delta-1|\wedge|-\delta-1|=||\delta|-1|$.
Moreover, (\ref{ben}) and the second statement of  (\ref{3s})  imply the third one.
\end{proof}
\section{Appendix: Further results from the literature}
The results in this section 
do not depend on those of Sections \ref{moon}--\ref{main}.
\begin{lemma}\label{billy}{\bf (\cite[Th.\ 3.2]{Bil99})}
Let $(S,d)$ be a metric spaces. Suppose that   $X_{m,n}, X_n,$ $Z_m\ (n,m\in\N)$ and $X$ are $S$-valued random variables such that $X_{n,m}$ and $X_n$ are defined on the same probability space with probability measure $P_n$ for all $n,m\in\N$. If
$X_{m,n}\begin{array}[t]{c}\Longrightarrow\vspace*{-3mm}\\{\scriptstyle n\to\infty}\end{array}
Z_m\begin{array}[t]{c}\Longrightarrow\vspace*{-3mm}\\{\scriptstyle m\to\infty}\end{array}X$
and 
\begin{equation}\label{jon}
\lim_{m\to\infty}\limsup_{n\to\infty}P_n\left[d(X_{m,n},X_n)>\eps\right]=0
\end{equation}
for each $\eps>0$ then
$X_{n}\begin{array}[t]{c}\Longrightarrow\vspace*{-3mm}\\{\scriptstyle n\to\infty}\end{array}
X$.
\end{lemma}
\begin{lemma}\label{sca}{\rm \bf (Scaling of squared Bessel processes)}
For $i=1,2$ let $x_i>0$ and let $Y_i=(Y_i(t))_{t\ge 0}$ solve the SDE {\rm (\ref{bes})} with initial condition $Y_i(0)=x_i$.
Then 
\[Y_1\left(\cdot\wedge \si_0^{Y_1}\right)\quad\mbox{and}\quad
\frac{x_1}{x_2}Y_2\left(\frac{x_2}{x_1}\cdot\wedge \si_0^{Y_2}\right)
\]
have the same distribution. 
\end{lemma}
\begin{proof} For $\delta\ge 0$ see \cite[Ch.\ XI, (1.6) Prop.]{RY99}. For general $\delta$  see \cite[A.3]{GY03} and \cite[Lem.\ 3.2 (i)]{KM11}. 
\end{proof}
\begin{lemma}\label{3.35} Let $\delta<1$ and
let $Y=(Y(t))_{t\ge 0}$ solve the SDE {\rm (\ref{bes})} with $Y(0)=1$. Then there are constants $a(\delta),b(\delta)\in(0,\infty)$ such that
\begin{eqnarray}\label{L3.3}
\lim_{t\to\infty}t^{1-\delta}P_1[\si_0^Y>t]&=&a(\delta)\qquad\mbox{and}\\
\lim_{t\to\infty}t^{(1-\delta)/2}P_1\left[\int_0^{\si_0^Y}Y(s)\ ds>t\right]&=&b(\delta).
\label{L3.5}
\end{eqnarray} 
\end{lemma}
\begin{proof}
Statements (\ref{L3.3}) and (\ref{L3.5}) follow from Lemma \ref{sca} and are the contents of  \cite[Lemma 3.3]{KM11} and 
\cite[Lemma 3.5]{KM11}, respectively.

In fact, with respect to (\ref{L3.3}) much more is known. In \cite[(15)]{GY03} a formula for the density of the first passage time to 0 of a Bessel process with dimension in $[0,2)$ is given. This formula implies (\ref{L3.3}) in the case $0\le \delta<1$. It has been noticed in the remark after 
\cite[(4.24)]{Ale11} that the same formula also holds for negative dimensions.
\end{proof}
\begin{lemma} {\rm (cf.\ \cite[Lemma A.1]{KM11})}
  \label{gld} Let $\xi_i,\ i\in\N,$ be independent
  random variables which are geometrically distributed with parameter 1/2 and  $E[\xi_i]=1$. Then  for
  all $x,y\in\N$,
  \[
    P\left[\left|\sum_{i=1}^x(\xi_i-1)\right|\ge y\right]\le
    2\exp\left(\frac{-y^2}{6(x\vee y)}\right).
  \]
\end{lemma}
\begin{proof} We are going to use a special case of Azuma's inequality, which states that for the simple symmetric RW $(S_n)_{n\ge 0}$ on $\Z$ starting at 0 and any $a,n\ge 0$, $P[S_n\ge a]\le \exp(-a^2/(2n))$, see e.g.\ \cite[Theorem A.1.1]{AS00}.
 Let $(Y_i)_{i\ge 1}$ be an independent sequence of Bernoulli(1/2)-distributed random variables. By interpreting $\xi_i+1$ as the time of the first appearance of ``heads'' in a sequence of independent fair coin flips we obtain
\begin{eqnarray*}
P\left[\sum_{i=1}^x(\xi_i-1)\ge  y\right]&=&
P\left[\sum_{i=1}^x(\xi_i+1)\ge 2x+y\right]
=P\left[\sum_{i=1}^{2x+y-1}Y_i< x\right]\\
&=&P\left[\sum_{i=1}^{2x+y-1}(2Y_i-1)< -(y-1)\right]\ =\ P\left[S_{2x+y-1}\ge y\right]\\
&\le& 
e^{-y^2/{(2(2x+y-1))}}\le e^{-y^2/(6(x\vee y))}
\end{eqnarray*}
by Azuma's inequality.
 Similarly, for $x\ge y$,
\begin{eqnarray}\label{non}
P\left[\sum_{i=1}^x(\xi_i-1)\le - y\right]&=&
P\left[\sum_{i=1}^x(\xi_i+1)\le 2x- y\right]
=P\left[\sum_{i=1}^{2x-y}Y_i\ge x\right]\\
&=&P\left[S_{2x-y}\ge y\right]\le e^{-y^2/{(2(2x-y))}}\le e^{-y^2/(6(x\vee y))}\nonumber
\end{eqnarray}
again by Azuma's inequality. For $x<y$ the quantities in (\ref{non}) are 0. A union bound now yields the claim.
\end{proof}
The following lemmas about the BPs $V,\hV$, and $\widetilde V$ are slight modifications of results from \cite{KM11}. 
The first one controls how much the BPs $V$ and $\hV$ ``overshoot'' $x$ at the times $\tau_x$ and $\si_x$.

\begin{lemma}\label{os}{\bf(Overshoot)}
There are constants $\con{o1},\con{o2}>0$ and $N\in\N$ such that for
every $x\ge N$, $y\ge 0$, and $\eps>0$, 
\begin{align}
  \max_{0< z<x}P_z\left[V_{\tau^V_x}>x+y\,|\,\tau^V_x<\sigma^V_0\right]&\le 
c_{\ref{o1}}( e^{-c_{\ref{o2}}y^2/x} + e^{-c_{\ref{o2}}y}),\label{over} \\ \max_{x<z< 4x} 
P_z[V_{\sigma^V_x\wedge \tau^V_{4x}}<x-y]&\le c_{\ref{o1}}
  e^{-c_{\ref{o2}}y^2/x},\qquad\mbox{and}\label{under}\\
\label{cos}
\max_{0<
  z<x}P_z\left[\tau^{\hV}_x<\infty,\ \hV_{\tau^{\hV}_x}>(1+\epsilon)x
\right]&\le c_{\ref{o1}}( e^{-c_{\ref{o2}}\epsilon^2x} +
e^{-c_{\ref{o2}}\epsilon x}).
\end{align}
\end{lemma}
\begin{proof} The proofs of the first two statements repeat those of \cite[Lemma 5.1]{KM11}. 
 For the third statement note that by definition of
  $\hV$, 
\begin{eqnarray*}
\lefteqn{P_z\left[\tau^{\hV}_x<\infty,\ \hV_{\tau^{\hV}_x}>(1+\epsilon)x
     \right]\ =\ P_z\left[\tau^V_x<\infty,\ V_{\tau^V_x}>(1+\epsilon)x\,|\,\sigma^V_0<\infty \right]}\\ 
&=&\frac{P_z\left[\sigma^V_0<\infty\,|\,
      \tau^V_x<\infty,\ V_{\tau^V_x}>(1+\epsilon)x\right]}{P_z\left[\sigma^V_0<\infty\right]}\,P_z\left[\tau^V_x<\infty,\ V_{\tau^V_x}>(1+\epsilon)x\right]\\
&\le&P_z\left[\tau^V_x<\infty,\ V_{\tau^V_x}>(1+\epsilon)x\right]
\end{eqnarray*} since by the strong Markov property and
  monotonicity with respect to the starting point, the fraction above
  does not exceed 1 for all $0<z<x$.
An
  application of (\ref{over}) with $y=\epsilon x$ completes the proof.
\end{proof}
\begin{lemma}
  \label{5.2} Let $\delta>1$, $a\in(1,2]$, $|x-a^n|\le a^{2n/3}$, and
  $\gamma:=\inf\{k\ge 0\,:\, V_k\not\in(a^{n-1},a^{n+1})\}$. Then for all
sufficiently large $n$,
\begin{align*}
  (&i)\quad P_x\Big[\mathrm{dist}(V_\gamma ,(a^{n-1},a^{n+1})) \ge
  a^{2(n-1)/3}\Big] \le \exp(-a^{n/4});\\(i&i)\quad
  \left|P_x[V_\gamma\leq a^{n-1}] - \frac{1}{a^{\delta-1} +1}\right| \leq
a^{-n/4}.
\end{align*}
\end{lemma}
\begin{proof}
  The proof is identical to that of \cite[Lemma 5.2]{KM11} if one
  replaces $\delta$ with $1-\delta$ throughout.
\end{proof} 
\begin{lemma}\label{5.3}Let $\delta>1, a\in(1,2]$, $\la\in(0,1/8)$, and 
\[h_\ell^\pm(m):=\prod_{r={\ell+1}}^m\left(a^{\delta-1}\mp a^{-\la r}\right)^{-1}\qquad\mbox{for all $\ell,m\in\N,\ \ell<m$}.
\]
Then there is $\ell_0\in\N$ such that if $\ell,m,x\in\N$ satisfy $\ell_0\le \ell<m$ and $|x-a^m|\le a^{2m/3}$ then
\begin{equation}\label{ip}
h_\ell^-(m)\le P_x[\si_{a^\ell}^V<\infty]\le
h_\ell^+(m).
\end{equation}
Moreover, there are $K_1,K_2:\N\to(0,\infty)$ such that $K_1(\ell), K_2(\ell)\to 1$ as $\ell\to\infty$ 
and
\begin{equation}\label{pp}
K_1(\ell)\le h_\ell^\pm(m)\ a^{(m-\ell)(\delta-1)}\le K_2(\ell)
\end{equation}
for all $m>\ell\ge 0$.
\end{lemma}
\begin{proof}
Inequalities (\ref{ip}) will
follow if we show that for all
$u>m$,
\[\frac{h_\ell^-(m)-h_\ell^-(u)}{1-h_\ell^-(u)}\le
P_x[\si_{a^\ell}^V<\tau_{a^u}^V]\le
\frac{h_\ell^+(m)-h_\ell^+(u)}{1-h_\ell^+(u)}\] 
and then let
$u\to\infty$. The inequalities in the previous line are equivalent
to 
\[\frac{1-h_\ell^+(m)}{1-h_\ell^+(u)}\le
P_x[\si_{a^\ell}^V>\tau_{a^u}^V]\le
\frac{1-h_\ell^-(m)}{1-h_\ell^-(u)}.\] 
The proof of the last
inequalities follows essentially word for word the one of
\cite[Lemma~5.3]{KM11} and uses Lemma~\ref{os} ((\ref{over}) and (\ref{under})), and
Lemma~\ref{5.2} in place of Lemma~5.1 and Lemma~5.2 of \cite{KM11},
respectively. We omit the details.

To see (\ref{pp}) we notice that $h_\ell^-(m)\le h_\ell^+(m)$ for all
$m>\ell$ and
\begin{align*}
  h_\ell^+(m)a^{(m-\ell)(\delta-1)}&=\prod_{r=\ell+1}^m (1-a^{-\lambda r-\delta+1})^{-1}\le
\prod_{r=\ell+1}^\infty (1-a^{-\lambda
  r})^{-1}=:K_1(\ell)<\infty;\\
h_\ell^-(m)a^{(m-\ell)(\delta-1)}&=
\prod_{r=\ell+1}^m (1+a^{-\lambda r-\delta+1})^{-1}\ge\prod_{r=\ell+1}^\infty (1+a^{-\lambda r})^{-1} =:K_2(\ell)>0.
\end{align*}
Clearly, $K_1(\ell), K_2(\ell)\to 1$ as $\ell\to\infty$.
\end{proof}
\begin{lemma}\label{5.5ab} Let $\delta<1$.
For every $\ga>0$ there is $\con{eps}(\ga)\in(0,\infty)$ such that 
$P_n[\tau_{c_{\ref{eps}}n}^{V}<\si_0^{V}]<\ga$ for all $n\in\N$. Moreover,  $P_n[\si_0^V<\infty]=1=P_n[\si_0^{\widetilde V}<\infty]$. Similarly, for
every  $\ga>0$ there is  $\con{eps2}(\ga)\in(0,\infty)$ such that  
$P_n[\tau_{c_{\ref{eps2}}n}^{\widetilde V}<\si_0^{\widetilde V}]<\ga$
for all $n\in\N$.
\end{lemma}
\begin{proof}
For the proof of the first claim about $V$ see the proof of \cite[(5.5)]{KM11} replacing $\delta$ with $1-\delta$ throughout and using Lemma~\ref{5.3} instead of \cite[Lem.\ 5.3]{KM11}.

For the second statement observe that $P_n$-a.s.\ $\si_0^V\wedge \tau_{c_{\ref{eps}}n}^{V}<\infty$ due to (\ref{ind}) and (\ref{gg}). This together with the first statement implies $P_n$-a.s.\ $\si_0^V<\infty$.

For the third assertion observe that due to (\ref{ind}), (\ref{a3}), and (\ref{gg}), $P_n[\si_0^{\widetilde V}=1]>0$ for all $n\ge 1$. Therefore, $P_n[\si_0^{\widetilde V}=\infty]\le P_n[\widetilde V_m\to\infty]$, which is equal to 0 since $\widetilde V=V$ up to time $\si_M^{\widetilde V}$ and therefore for all $i\ge 1$, $P_i[\si_M^{\widetilde V}<\infty]=P_i[\si_M^{V}<\infty]\ge P_i[\si_0^{V}<\infty]=1$
as we have already shown.

For the final statement let $t=c_{\ref{eps}}(\ga/2)$. Then
\begin{eqnarray*}
\limsup_{n\to\infty}P_n\left[\tau_{tn}^{\widetilde V}<\si_0^{\widetilde V}\right]&\le& \limsup_{n\to\infty}P_n\left[\tau_{tn}^{\widetilde V}<\si_M^{\widetilde V}\right]+\limsup_{n\to\infty}P_n\left[\si_M^{\widetilde V}<\tau_{tn}^{\widetilde V}<\si_0^{\widetilde V}\right]\\
&\le& \limsup_{n\to\infty}P_n\left[\tau_{tn}^{V}<\si_M^{V}\right]+
\max_{i=1,\ldots,M}\limsup_{n\to\infty}P_i\left[\tau_{tn}^{\widetilde V}<\si_0^{\widetilde V}\right]\\
&\le&\ga/2+\max_{i} P_i[\si_0^{\widetilde V}=\infty].
\end{eqnarray*}
Since $P_i[\si_0^{\widetilde V}=\infty]=0$ for all $i$ 
this implies the claim. 
\end{proof}
\begin{cor}{\rm \bf (Extinction and survival of $V$)}\label{rc} 
If $\delta<1$ then $P_1[\si_0^V<\infty]=1$.
If $\delta>1$ then $P_1[\si_0^V=\infty]>0$ and for all $N\ge 0$, $P_n[\si_N^V<\infty]\to 0$ as $n\to\infty$.
\end{cor}
\begin{proof}
The first claim is contained in Lemma \ref{5.5ab}. Lemma \ref{5.3} implies that  $P_n[\si_N^V<\infty]\to 0$ as $n\to\infty$. Choose $n\in\N$ such that $P_n[\si_0^V<\infty]<1$. Then
$P_1[\si_0^V=\infty]\ge P_1[V_1\ge n, \si_0^V=\infty]\ge P_n[\si_0^V=\infty]>0$.
\end{proof}
As a result we recover, except for the critical case $|\delta|=1$, the recurrence-transience criterion \cite[Th.\ 1]{KZ08} in a more self-contained way without using \cite{FYK90}.
\begin{cor}{\rm \bf (Recurrence and transience of ERW)}\label{rc2} 
If $|\delta|<1$ then the ERW $X$ returns a.s.\ infinitely often to its starting point. If $\delta>1$ then $X_n\to\infty$ a.s., whereas
$X_n\to-\infty$ a.s.\ if $\delta<-1$.
\end{cor}
\begin{proof}
The proof goes along the lines of proof of \cite[Th.\ 1]{KZ08} except that \cite[Prop. 9]{KZ08} (for $|\delta|\ne 1$) now follows from Corollary \ref{rc} and $P_1[T_0<\infty]=P_1[\si_0^V<\infty]$, see (\ref{uk}) and (\ref{knopf}). 
\end{proof}

{\bf Acknowledgment:} E. Kosygina was partially
supported by the Simons Foundation, Collaboration Grant for
Mathematicians \# 209493 . M.\ Zerner's work was supported by the
European Research Council, StG 208417-NCIRW.

\vspace*{2mm}

{\sc \small
\begin{tabular}{ll}
Department of Mathematics& \hspace*{20mm}Mathematisches Institut\\
Baruch College, Box B6-230& \hspace*{20mm}Universit\"at T\"ubingen\\
One Bernard Baruch Way&\hspace*{20mm}Auf der Morgenstelle 10\\
New York, NY 10010, USA&\hspace*{20mm}72076 T\"ubingen, Germany\\
{\verb+elena.kosygina@baruch.cuny.edu+}& \hspace*{20mm}{\verb+martin.zerner@uni-tuebingen.de+}
\end{tabular}
}

\end{document}